\documentclass[11pt,oneside]{amsart}
\usepackage{amsfonts, amssymb}
\usepackage{amsmath,amsthm,amscd,xypic}
\usepackage{color}
\usepackage{epic,eepic}
\usepackage{geometry}
\usepackage{graphicx}
\newtheorem{theorem}{Theorem}[section]
\newtheorem{thm}[theorem]{Theorem}
\newtheorem*{thm*}{Theorem}
\newtheorem{lem}[theorem]{Lemma}
\newtheorem{cor}[theorem]{Corollary}
\newtheorem*{cor*}{Corollary}
\newtheorem{prop}[theorem]{Proposition}

\newtheorem*{thmA}{Theorem A}

\theoremstyle{definition}

\newtheorem{example}[theorem]{Example}

\newtheorem*{conj*}{Conjecture}
\newtheorem{conj}{Conjecture}

\newtheorem{remark}[theorem]{Remark}

\numberwithin{equation}{section}

\newcommand{\vol}[1]{\text{vol}\left(#1\right)}
\newcommand{\lh}[1]{\mathcal{L}\left(#1\right)}

\newcommand{\intl}[1]{\textsf{int}\left(#1\right)}

\newcommand{\snk}{\text{sn}_{\kappa}}
\newcommand{\diam}{\text{diam}}
\newcommand{\Alex}{\text{Alex\,}}

\newcommand{\Alexnk}{\text{Alex}^n(\kappa)}
\newcommand{\dsp}{\displaystyle}
\newcommand{\geod}[1]{[\,#1\,]}
\newcommand{\geodii}[1]{]#1[\,}
\newcommand{\geodci}[1]{[\,#1[\,}
\newcommand{\geodic}[1]{\,]#1\,]}
\newcommand{\ang}[3]{\measuredangle\left({#1}\,_{#3}^{#2}\right)}
\newcommand{\tang}[4]{\tilde\measuredangle_{#1}\left({#2}\,_{#4}^{#3}\right)}

\newcommand{\drn}{\uparrow}
\newcommand{\scap}[2]{\underset{#1}{\overset{#2}\cap}}
\newcommand{\scup}[2]{\underset{#1}{\overset{#2}\cup}}
\newcommand{\dsju}[2]{\underset{#1}{\overset{#2}\amalg}}
\newcommand{\ssum}[2]{\underset{#1}{\overset{#2}\Sigma}\,}
\newcommand{\reg}[1]{{#1}^{\text{\,Reg}}}

\begin{document}



\title{Lipschitz-Volume rigidity in Alexandrov geometry}

\author{Nan Li}
\address{Department of Mathematics, The Penn State University, State College, PA 16802}
\email{lilinanan@gmail.com}
\thanks{The author was partially supported by the research funds managed by Penn State University.}



\date{\today}



\begin{abstract}
  We prove a Lipschitz-Volume rigidity theorem in Alexandrov geometry, that is, if a 1-Lipschitz map $f\colon  X=\amalg X_\ell\to Y$ between Alexandrov spaces preserves volume, then it is a path isometry and an isometry when restricted to the interior of $X$. We furthermore characterize the metric structure on $Y$ with respect to $X$ when $f$ is also onto. This implies the converse of Petrunin's Gluing Theorem: if a gluing of two Alexandrov spaces via a bijection between their boundaries produces an Alexandrov space, then the bijection must be an isometry.
\end{abstract}

\maketitle

\tableofcontents

\section*{Introduction}

In this paper we study the rigidity properties of length metric spaces $X$ and $Y$ whenever there exists a 1-Lipschitz and volume (Hausdorff measure) preserving map $f\colon X\to Y$. If $X$ and $Y$ are both compact manifolds and have the same dimension, it is not difficult to prove that such map $f$ must be an isometry. We call this property {\it Lipschitz-Volume rigidity} (abbreviated as LV-rigidity). This property does not hold for general singular spaces, even if $Y$ contains only one singular point (see Example \ref{eg.cube.edge}). In general, the map $f$ may not yield any rigid relationship between the metrics on $X$ and $Y$ near the singular points.

In this paper we establish a LV-rigidity theorem for Alexandrov spaces. In \cite{Li13-vol.rig.Ric}, N. Li and F. Wang prove a similar result for non-collapsed Gromov-Hausdorff limits of manifolds with Ricci curvature uniformly bounded from below. The LV-rigidity holds with a lower curvature bound partially because the regular parts in these spaces are (almost) convex (see \cite{BGP}, \cite{CC97-I}, \cite{CN12} and compare to Example \ref{eg.cube.edge}). The Alexandrov case is more complicated because of the existence of boundary. These rigidity results provide a way to classify the length metric on $X$ or $Y$ when its volume attains a relatively minimal or maximal value (see \cite{BGCS}, \cite{LR10}, \cite{Storm02}, \cite{Storm06} and Section 5 in our paper).

Now we state our main results. By ``vol'', we denote the top dimensional Hausdorff measure. Due to \cite{LR-JDG}, these results also hold for rough volume. By $(\amalg X_\ell, d)$ we denote the disjoint union of length metric spaces $\{(X_\ell,d_\ell)\}$, where $d(p,q)=d_\ell(p,q)$ if $p,q\in X_\ell$ for some $\ell$ and $d(p,q)=\infty$ otherwise. Let $\intl{X}$ denote the interior of $X$ and $\partial X=X\setminus\intl{X}$ be the boundary. A map $f\colon X\to Y$ is said to be a path isometry if it preserves the lengths of paths. Let $|A|$ denote the cardinality of a set $A$.


\begin{thmA}[LV-rigidity]
  Let $Y$ and $X_\ell$, $\ell=1,2,\dots,N_0$ be $n$-dimensional Alexandrov spaces. If a 1-Lipschitz map $\dsp f\colon X=\dsju{\ell=1}{N_0} X_\ell\to Y$ satisfies $\vol X=\vol {f(X)}$, then $f$ is a path isometry and $f|_{\intl{X}}$ is an isometry with respect to the intrinsic metrics. If $f$ is also onto, then the metric on $Y$ coincides with the metric induced by the gluing structure $x_1\sim x_2$ $\Leftrightarrow$ $f(x_1)=f(x_2)$. Moreover,
  \begin{enumerate}
    \renewcommand{\labelenumi}{(A.\arabic{enumi})}
      \item for any $y\in Y$,  $\big|f^{-1}(y)\big|<\infty$;
      \item for any $y\in Y$, if $\big|f^{-1}(y)\big|\ge 2$, then $f^{-1}(y)\subseteq\partial{X}$;
      \item for any point $p\in Y$ such that $\big|f^{-1}(p)\big|\ge2$ and any $r>0$, the Hausdorff dimension of $\left\{y\in B_r(p): \big|f^{-1}(y)\big|=2\right\}$ is exactly $n-1$;
      \item the Hausdorff dimension of $\left\{x\in X: |f^{-1}(f(x))|\ge 3\right\}$ is at most $n-2$.
\end{enumerate}
\end{thmA}

It follows immediately from Theorem A that for every $\ell$, $f|_{\intl{X_\ell}}$ is an isometry with respect to the intrinsic metrics. Suppose that $\gamma_1, \gamma_2\subset X$ are glued, that is, $f(\gamma_1(t))=f(\gamma_2(t))$ for all $t$, where $\gamma_i(t)$ are parameterized by arc lengths. By Theorem A, the lengths of curves $\lh{\gamma_1}=\lh{f(\gamma_1)}=\lh{f(\gamma_1)}=\lh{\gamma_2}$. In this sense, we say that $\{X_\ell\}$ are glued by isometry. By (A.3) and (A.4), the gluing structure is uniquely determined by the part of one-to-one gluing. The gluing along non-extremal subset is possible (see Example \ref{eg.non.extremal}).
When $N=1$, we call the gluing {\it self-gluing} (see Examples \ref{eg1} and \ref{eg2}). In this case, Theorem A shows that without losing volume or increasing the distance, the metric on an Alexandrov space is ``rigid'' up to an isometric boundary gluing.
%
We show that $f$ is in fact an isometry in some special cases. The proofs are included in the proof of Corollary \ref{thmC.vol.iso}.

\begin{cor}\label{shrinking.cor}
  Under the assumptions as in Theorem A and that $f$ is onto, if any of the following is satisfied then $X$ is connected ($N_0=1$) and $f$ is an isometry.
  \begin{enumerate}
  \item $f$ is injective.
  \item The boundary $\partial X_\ell=\varnothing$ for some $\ell$.
  \item $f$ is surjective and $f(\partial X_\ell)\subseteq\partial Y$ for all $\ell$.
  \end{enumerate}
\end{cor}

As a special case of Theorem A, we give an affirmative answer to the following conjecture in \cite{LR10}.

\begin{conj*}
  Let $X$ and $Y$ be $n$-dimensional Alexandrov spaces and $f\colon X\to Y$ be a distance non-increasing onto map. If $\text{vol}(X)=\text{vol}(Y)$, then $Y$ is isometric to a space glued from $X$ along $\partial X$ and $f$ is a path isometry. In particular, $X$ is isometric to $Y$ if $\partial X=\varnothing$ or $f$ is injective.
\end{conj*}

Theorem A can be used to classify Alexandrov spaces which have relatively maximum or minimum volume (see \cite{Storm02}, \cite{Storm06} and Section 5).
Note that a gluing structure (see Section 1 for the rigorous definition) naturally induces a 1-Lipschitz onto map, which is the projection map. If the gluing is along a lower dimensional subset, then the volume is preserved by such projection map. Thus, Theorem A describes some necessary conditions to produce Alexandrov spaces by gluing Alexandrov spaces without losing volume. Recall the following well known theorem.

%
%

\begin{thm*}[Petrunin, \cite{Pet97}]\label{pet.glu}
The gluing of two Alexandrov spaces via an isometry between their boundaries produces an Alexandrov space with the same lower curvature bound.
\end{thm*}


The following result, conjectured by A. Petrunin, follows from Theorem A and the above theorem.


\begin{thm}\label{pet.iff}
  Assume that $n$-dimensional Alexandrov spaces $X_1$ and $X_2$ are glued via an identification $x\sim \phi(x)$, where $\phi:\partial X_1\to\partial X_2$ is a bijection. Then the glued space $Y=X_1\amalg X_2/x\sim \phi(x)$ is an Alexandrov space if and only if $\phi$ is an isometry with respect to the intrinsic metrics of $\partial X_1$ and $\partial X_2$.
\end{thm}

In general, gluing by isometry, together with the conditions (A.1)--(A.4) are not sufficient to guarantee that the glued space is an Alexandrov space (see Example \ref{eg.non.extremal}). As a generalization of Theorem \ref{pet.iff}, we conjecture that


\begin{conj}
  {\it A gluing along lower dimensional subsets of $n$-dimensional Alexandrov spaces produces an Alexandrov space if and only if the gluing is by isometry and the tangent cones over the glued space are Alexandrov spaces with curvature bounded from below by $0$.}
\end{conj}






The starting point of our proof is to show that $f$ is injective and almost preserves the lengths of paths when restricted to the set of $(n,\delta)$-strained (regular) points. The main difficulty is to show that the continuous extension of the restricted map is a path isometry. This substantially relies on the local Alexandrov structure, but fails for general singular spaces (see Example \ref{eg.cube.edge}). For an arbitrary Lipschitz curve $\gamma\subset \intl X$, using a technique of perturbation, we find an approximation $\sigma_i\subset f(X^\delta)$ for $f(\gamma)$ which satisfies
\begin{enumerate}
  \item $\sigma_i\to f(\gamma)$ uniformly;
  \item $\lh{\sigma_i}\to \lh{f(\gamma)}$;
  \item $f^{-1}(\sigma_i)\to\gamma$.
\end{enumerate}
Then by the semi-continuity of the lengths of curves, we conclude that $f|_{\intl X}$ is length preserving. In the case of $\gamma\subset \partial X$, a different approach is needed because $f(\intl X)$ needs not to be convex in $Y$ (see Example \ref{eg.bgy.conv}).

We divide the paper into five sections. In the first section we begin with basic definitions and a restatement of Theorem A. We will also outline our proof and discuss some examples.
In {\it Section 2}, we recall some results for Alexandrov spaces mainly from \cite{BGP} and \cite{OS94}. {\it Section 3} is aimed to show that $f$ is an isometry when restricted to the interior (Lemma \ref{thmC.int.iso}). 
We complete the proof of Theorem \ref{main.thm} in {\it Section 4} 
and give some applications in {\it Section 5}, including the LV-rigidity of spaces of directions (Theorem \ref{shrink.spd}). With a weaker assumption, we proves the relatively almost maximum volume theorems in \cite{LR10}, which appears as a natural extension of Grove and Petersen's results (\cite{GP92}) in Riemannian geometry.

I would like to thank Stephanie Alexander, Richard Bishop, Vitali Kapovitch, Anton Petrunin for their interest. The final version benefitted from numerous discussions with Jianguo Cao, Karsten Grove and Xiaochun Rong. It is my pleasure to thank the referee for many useful suggestions to improve the quality of the paper.

\medskip

{\bf Conventions and Notations}
\begin{itemize}
  \item $\Bbb S_\kappa^n$: the $n$-dimensional space form with constant curvature $\kappa$.
  \item $\dim_H(A)$:\; the Hausdorff dimension of $A$.
  \item $\vol A$:\; the $n$-dimensional Hausdorff measure of $A$, where $n$ is the Hausdorff dimension of $A$.
  \item $\intl A$:\; the interior of $A$.
  \item $\partial A$:\; the boundary of $A$.
  \item $d_A(x,y)$ or $|xy|_A$:\; the distance between two points $x$ and $y$ with respect to the intrinsic length metric over $A$.
  \item $B_r(p,X)=\{x: |px|_X<r\}$:\; the metric ball in $X$ around the point $p$. Sometimes, we omit $X$ if this does not cause any ambiguity. We use $o$ for the center of the ball if it is not necessarily to be specified. For instance, we denote a ball in $\Bbb S_\kappa^n$ by $B_r(o,\Bbb S_\kappa^n)$.
  \item $\geod{pq}_X$:\; a minimal geodesic connecting points $p$ and $q$ in $X$. Once it appears, it will always mean the same geodesic in the same context.
  \item $\drn_p^q$:\; the unit tangent vector of a geodesic $\geod{pq}$ at $p$.
  \item $\drn_{\geod{pq}}$:\; the unit tangent vector of the given geodesic $\geod{pq}$.
  \item $\lh\gamma$:\; the length of the curve $\gamma\colon[a,b]\to X$, defined as
      $$\lh\gamma
        =\sup\left\{\sum_{i=1}^Nd(\gamma(t_i),\gamma(t_{i+1})) : a=t_0<t_1<\dots<t_n=b\right\}
      .$$
  \item $\tau(\delta)$:\; a positive function in $\delta$ that satisfies $\dsp\lim_{\delta\to 0^+}\tau(\delta)=0$. 
  \item $X_n\overset{d_{GH}}{\longrightarrow}X$:\; the sequence of length metric spaces $X_n$ Gromov-Hausdorff converges to $X$.
  \item Let $A\subset X$ and $f\colon X\to Y$. The restricted map $f|_A$ is called an isometry if $|ab|_A=|f(a)f(b)|_{f(A)}$ for any $a,b\in A$.
\end{itemize}

By $\Alexnk$ we denote the isometric class of $n$-dimensional Alexandrov spaces with curvature $\ge\kappa$. For $X\in\Alexnk$, we use the following notations (c.f. \cite{BGP}).
\begin{itemize}
    \item $T_p(X)$ (or sometimes $T_p$):\; the tangent cone at point $p\in X$.
    \item $O_p$ :\; the vertex of the tangent cone $T_p(X)$.
    \item $\Sigma_p(X)$ (or sometimes $\Sigma_p$):\; the space of directions at a point $p\in X$.
    \item $X^\delta=X^{(n,\delta)}$:\; the set of $(n,\delta)$-strained points in $X$.
    \item $\snk(t)
        =\begin{cases}
          \frac{1}{\sqrt{\kappa}}\sin(\sqrt{\kappa}\,t), &\text{ if } \kappa>0;\\
          t, &\text{ if } \kappa=0;\\
          \frac{1}{\sqrt{-\kappa}}\sinh(\sqrt{-\kappa}\,t), &\text{ if } \kappa<0.
        \end{cases}$
\end{itemize}

\section{Lipschitz-Volume rigidity and examples}

We begin with a rigorous definition of the gluing of length metric spaces (c.f. \cite{BBI} $\S3$). Let $\{(X_\ell,d_\ell)\}$ be a collection of compact length metric spaces and $(X,d)=\amalg  (X_\ell,d_\ell)$. Let $\mathcal{R}$ be an equivalence relation over $X$,  denoted as $p\overset {\mathcal{R}}\sim q$. The quotient pseudometric $d_{\mathcal{R}}$ on $X$ is defined as
$$d_{\mathcal{R}}(p,q)
  =\inf\left\{\sum_{i=1}^Nd(p_i,q_i):p_1=p, q_N=q, p_{i+1}\overset {\mathcal{R}}\sim q_i,
  N\in\Bbb N
\right\}.$$
By identifying the points with zero $d_{\mathcal{R}}$-distance, we obtain a length metric space $(X/d_{\mathcal{R}}, \bar d_{\mathcal{R}})$, which is called the glued space from $\{X_\ell\}$ along the equivalence relation $d_{\mathcal{R}}$. The projection map $f\colon X\to Y$ is a 1-Lipschitz onto. 


Now we let $X=\dsju {\ell=1}{N_0}  X_\ell$ be the disjoint union of compact Alexandrov spaces $X_\ell\in\Alexnk$, $\ell=1,2,\dots,N_0$. Let $f\colon X\to Y$ be a 1-Lipschitz onto map which preserves volume, that is, $\vol Y=\vol X=\ssum{\ell=1}{N_0}\vol{X_\ell}$.
For any $y\in Y$, the points in $f^{-1}(y)$ are identified (glued) in terms of the equivalence relation $x_1\sim x_2\Leftrightarrow f(x_1)=f(x_2)$. We give a stratification of the gluing points. Namely, let
$$G_Y^m=\left\{y\in Y,\; \left|f^{-1}(y)\right|=m\right\}
  \text{\quad and \quad} G_X^m=f^{-1}(G_Y^m).
$$
We call $m_0=\max\{m:G_Y^m\neq\varnothing\}$ the maximum gluing number. In general, $m_0$ is independent of $N_0$. In our case, we will show that
$$m_0\le C\left(n,\kappa,\max_{1\le\ell\le N_0}\{\text{diam}(X_\ell)\}, \min_{1\le\ell\le N_0}\{\vol{X_\ell}\}\right)<\infty.$$
We denote the sets of gluing pints by $G_Y=\scup{m=2}{m_0} G_Y^m$ and $G_X=\scup{m=2}{m_0} G_X^m$.

\begin{thm}\label{main.thm}
  Let $Y\in\Alexnk$ and $X=\dsju{\ell=1}{N_0}X_\ell$, where $X_\ell\in\Alexnk$, $\ell=1,\dots,N_0$. If there exists a 1-Lipschitz onto map $f\colon X\to Y$ and $\vol Y=\vol X$, then the following statements hold.
  \begin{enumerate}
    \item If $G_X\neq\varnothing$ then $G_X\subseteq\partial X$.
    \item $f$ is a path isometry and $f|_{\intl{X}}$ is an isometry.
    \item $m_0\le\frac{\vol{B_{d_0}(S_{\kappa}^n)}}{v_0}$, where $d_0=\underset{1\le\ell\le N_0}\max\{\text{diam}(X_\ell)\}$ and $v_0=\underset{1\le\ell\le N_0}\min\{\vol{X_\ell}\}$.
    \item If $G_X\neq\varnothing$, then for any point ${a}\in G_X$, $p=f({a})\in G_Y$ and $r>0$,
        $$\dim_H(B_r({a})\cap G_X^2)=\dim_H(B_r(p)\cap G_Y^2)=n-1$$ and  $$\dim_H\left(\scup{m=3}{m_0}G_X^m\right) =\dim_H\left(\scup{m=3}{m_0}G_Y^m\right)
        \le n-2.$$
  \end{enumerate}
\end{thm}

Note that the proof of Theorem \ref{main.thm} (1) and (2) also applies to the case that $f$ is not onto.
Theorem A follows from the Theorem \ref{main.thm} and its proof.

\begin{remark}\label{remark.thmC}
\quad
  \begin{enumerate}
    \item If we replace $f\colon X\to Y$ by a distance non-decreasing map $g\colon Y\to X$ and assume that $\vol Y=\vol X$, then $g^{-1}$ can be extended to a 1-Lipschitz onto map $f\colon X\to Y$ due to the compactness of $Y$ and $X_\ell$.
    \item By cutting $Y=\Bbb S_1^n$ into $m_0$ isometric petals $\{X_\ell\}$, we see that the estimate in Theorem \ref{main.thm} (3) is sharp for the gluing of multiple spaces. In this case, $\dim_H(G_X^2)=\dim_H(G_Y^2)=n-1$, $G_X^m=\varnothing$ for $3\le m\le m_0-1$ and $\dim_H(G_X^{m_0})=\dim_H(G_Y^{m_0})=n-2$.
  \end{enumerate}
\end{remark}

In some special cases, $X$ must be connected ($N_0=1$) and $f$ is in fact a global isometry.

\begin{cor}\label{thmC.vol.iso}
  Under the assumptions as in Theorem \ref{main.thm}, if any of the following is satisfied, then $N_0=1$ and $f$ is an isometry.
  \begin{enumerate}
  \item $\partial X_\ell=\varnothing$ for some $\ell$.
  \item $G_X=\varnothing$.
  \item $G_Y\subseteq\partial Y$.
  \item $f(\partial X)\subseteq \partial Y$.
  \item $f^{-1}(Y^\delta)\cap G_X=\varnothing$ for $\delta>0$ small.
  \item $f^{-1}(Y^\delta)\subseteq \intl X$ for $\delta>0$ small.
  \end{enumerate}
\end{cor}

\begin{proof}
(1) and (2) are immediate consequences of Theorem \ref{main.thm} (2). (4) follows from (3) and (6) follows from (5) since $G_X\subseteq\partial X$. We first prove (3). By the assumption, we have $f^{-1}(\intl Y)\subseteq X\setminus G_X$. By Theorem \ref{main.thm} (2), $f|_{X\setminus G_X}$ is an isometry. Therefore, $f$ is an isometry since $\intl Y\subseteq f(X\setminus G_X)$ is totally geodesic in $Y$.

To prove (5), it suffices to show that the condition in (5) implies the condition in (3). For any $y\in G_Y$, by Theorem \ref{main.thm} (4), we have
$$\dim_H(B_r(y)\cap G_Y)=n-1.$$
By the assumption, we see that $G_Y\cap Y^\delta=\varnothing$. Then $B_r(y)\cap G_Y\subseteq B_r(y)\setminus Y^\delta$ and thus
$$\dim_H\left(B_r(y)\setminus Y^\delta\right)\ge\dim_H(B_r(y)\cap G_Y)=n-1.$$
If $y\notin\partial Y$, then there is $r>0$ so that $B_r(y)\subset \intl Y$. Thus by \cite{BGP} 10.6.1 (see Corollary \ref{bgp10.6.1}),
$$\dim_H\left(B_r(y)\setminus Y^\delta\right)\le\dim_H\left(\intl Y\setminus Y^\delta\right)\le n-2,$$
a contradiction.
\end{proof}


If $Y$ is a space glued from $\{X_\ell\}$ and $y\in Y$ is a glued point, it is not surprising that the space of directions $\Sigma_y(Y)\in\Alex^{n-1}(1)$ is also glued from $\Sigma_{x_j}(X_{\ell_j})$, where $\{x_j\in X_{\ell_j}\}=f^{-1}(y)$. Let $\Sigma_{f^{-1}(y)}=\dsju j{}\Sigma_{x_j}(X_{\ell_j})$ denote the disjoint union of these spaces of directions.

\begin{thm}[Gluing of spaces of directions]\label{thmC.spd.glue}
  Under the assumptions as in Theorem \ref{main.thm}, for any $y\in Y$, $\Sigma_y$ is a space glued from $\Sigma_{f^{-1}(y)}$ without losing volume. Therefore, the gluing along the spaces of directions also satisfies Theorem \ref{main.thm}.
\end{thm}


Now we outline our proof of Theorem \ref{main.thm}. Clearly, for any $x\in X$, $y\in Y$ and $r>0$, we have $\vol{B_r(y)}=\vol{f^{-1}(B_r(y))}$ and $\vol{B_r(x)}=\vol{f(B_r(x))}$. Using this, we first show that $f(X^\delta)\subseteq Y^{\tau(\delta)}$ (Lemma \ref{delta.img}), $X^\delta\cap G_X=\varnothing$ (Lemma \ref{delta.inj}) and prove that $f|_{X^\delta}$ is a $\tau(\delta)$-almost isometry (Lemma \ref{delta.lip}). The path isometry does not follow from an easy continuous extension of Lemma \ref{delta.lip}. In general, when the geodesics converge, their lengths may not converge to the length of the limit geodesic. See the following example.

\begin{example}\label{eg.cube.edge}
Let $(X,d)$ be an $n$-dimensional compact Riemannian manifold. Let $p,q\in X$ and $\geod{pq}$ be a geodesic joining $p$ and $q$. Let $\mathcal L$ be the length of curves with respect to metric $d$. For an arbitrary $\lambda\in[0,1)$, we define another length metric $d_\lambda$ on $X$, which is induced by the following length structure:
$${\mathcal L}_\lambda(\sigma)
  =\lambda\cdot\lh{\sigma\cap\geod{pq}}+\lh{\sigma\setminus\geod{pq}}.
$$
Let $f\colon (X,d)\to (X,d_\lambda)$ be the identity map. Clearly, $f$ is 1-Lipschitz onto and volume preserving, but $f$ is not a path isometry. When $\lambda=0$, we have ${\mathcal L}_0(\sigma)
=\lh{\sigma\setminus\geod{pq}}$. This is equivalent to identifying $\geod{pq}$ as one point. In this case $(X,d_\lambda)$ has only one singular point. If $\lambda>0$, then $f$ is 1-Lipschitz and $\frac1\lambda$-co-Lipschitz. In this case, the singular set in $(X,d_\lambda)$ is $f(\geod{pq})$ which has dimension 1.

Note that $(X,d_\lambda)$ does not satisfy the dimension comparison lemma \ref{map.dim} due to the branching of $d_\lambda$-geodesics near $\geod{pq}$. For any curves $\gamma_i\in X$ such that $\gamma_i\to \geod{pq}$ and $\lh{\gamma_i\cap \geod{pq}}=0$, we have
$$\liminf_{i\to \infty}\,{\mathcal L}_\lambda(\gamma_i)
=\liminf_{i\to \infty}\,\lh{\gamma_i}\ge \lh{\geod{pq}}
>{\mathcal L}_\lambda(\geod{pq}).$$

\end{example}

We prove the path isometry as follows. Using the almost isometry $f|_{X^\delta}$, we show that in any small neighborhood of a glued point $x\in G_X$, the set of singular points has dimension at least $n-1$ (Lemma \ref{glue.dim}). Because $\dim_H(\intl X\setminus X^\delta)\le n-2$, we conclude that $G_X\subseteq\partial X$. Thus $f(\intl X)=Y\setminus f(\partial X)$ is open.
%
%
%
%
%
%
%
Using this and the dimension comparison lemma \ref{map.dim}, for any Lipschitz curve $\gamma\subset \intl{X_\ell}$, we are able to find a perturbation $\sigma_i$ of $f(\gamma)$ that satisfies
\begin{enumerate}
  \item $\sigma_i\subset f(X_\ell^\delta)$ (this implies that $\lh{f^{-1}(\sigma_i)}=(1+\tau(\delta))\cdot \lh{\sigma_i}$);
  \item $\sigma_i\to f(\gamma)$ uniformly (this and (1) imply $f^{-1}(\sigma_i)\to\gamma$ because $G_X\subseteq\partial X$);
  \item $\lh{\sigma_i}\to \lh{f(\gamma)}$.
\end{enumerate}
Then by the semi-continuity of lengths of curves, we get the following lemma.



\begin{lem}[Interior isometry]\label{thmC.int.iso}
  Let the assumption be as in Theorem \ref{main.thm}. Then $G_X\subseteq \partial X$ and $f|_{\intl X}$ is an  isometry.
\end{lem}

Using the isometry on the interior, we then construct a volume preserving map over the spaces of directions (see (\ref{df.vol})). This enables us to improve Lemma \ref{glue.dim} and prove the path isometry.

%
%



\begin{lem}[Path isometry]\label{thmC.L.pres}
  Let the assumption be as in Theorem \ref{main.thm}. Then $f$ is a path isometry.
\end{lem}


We conclude this section with some directive examples.

\begin{example}[Non-Alexandrov gluing]\label{eg.not.Alex}
We construct some examples (see Figure 1) of gluing that do not produce Alexandrov spaces from Alexandrov spaces. These examples are all in dimension $2$ and the gluing are all by isometry. The glued spaces are not Alexandrov spaces because one can find bifurcated geodesics near the glued points. In $(a)$, rectangle $ABCD$ is glued with rectangle $EFGH$ along an interior segment $\overline{EF}\subset ABCD$ and the edge $\overline{E'F'}\subset EFGH$. This gluing does not satisfy (A.2). In $(b)$, square $ABCD$ is glued with square $A'B'C'D'$ at the vertices $A\sim A'$. This gluing does not satisfy (A.3). In (c), three rectangles with the same size are glued along one common edge. This gluing does not satisfy (A.4).

\bigskip

\begin{center}\includegraphics[scale=1]{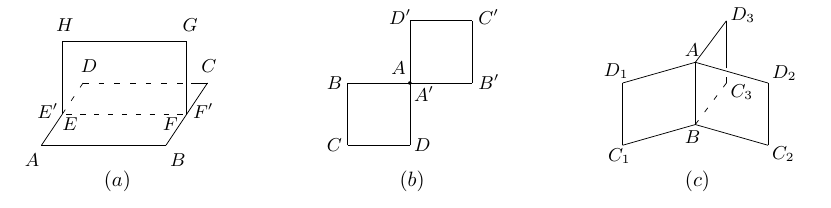}\end{center}

\begin{center} Figure 1\end{center}

\end{example}

\begin{example}[Non-isometric gluing]\label{eg.non.iso}

Let $E_r$ denote the 2-dimensional Euclidean square with side length $r$. Consider the boundary gluing of $E_R$ and $E_r$. Let $\phi:\partial E_R\to\partial E_r$ be a bijection that satisfies
$$\frac{1}{r}|\phi(x_1)\phi(x_2)|_{E_r}=\frac 1R|x_1x_2|_{E_R}.$$
Let $Y=E_R\amalg E_r/x\sim\phi(x)$ be the glued space along the identification $x\sim\phi(x)$. If $R=r$, then $\phi$ is an isometry, and thus $Y$ is an Alexandrov space as a doubled square. On the other hand, Theorem \ref{pet.iff} concludes that if $Y\in\Alex^2(\kappa)$, then $\phi$ has to be an isometry, that is, $R=r$. In fact, if $R>r$, let $f\colon E_R\amalg E_r\to Y$ be the projection map. Consider points $a, b\in f(\partial E_R)$. For a point $c\in f(\intl{E_R})$ close to $b$, geodesics $\geod{ab}_Y$ and $\geod{ac}_Y$ have overlaps, which yields a bifurcated geodesic.

One can also construct similarly the boundary gluing of two disks with radius $R$ and $r$. By Theorem \ref{pet.iff}, such gluing produces an Alexandrov space if and only if $R=r$. Note that in the disk case, even if $R\neq r$, there does not exist a bifurcated geodesic. 


\end{example}

\begin{example}[Involutional self-gluing]\label{eg1} This is an example for self-gluing (c.f. \cite{GP92}). Let $X=\Bbb D^2$ be a 2-dimensional flat unit disk. Then $\partial X=\Bbb S_1^1$  is a unit circle. Let $\phi:\partial X\to\partial X$ be an onto map and $Y=X/x\sim\phi(x)$ be the glued space. By Theorem \ref{rel.max}, $Y$ is an Alexandrov space if and only if $\phi$ is a reflection, antipodal map or identity, and $Y$ is homeomorphic to $\Bbb S^2$, $\Bbb {RP}^2$ and $\Bbb D^2$ respectively. Thus the maximum gluing number of the self-gluing over $\Bbb D^2$ is no more than $2$. However, if one estimates it by Theorem \ref{main.thm} (3),
$$m_0\le \frac{\pi\cdot 2^2}{\pi\cdot 1^2}=4.$$
\end{example}

\begin{example}[Gluing along non-extremal subsets]\label{eg.non.extremal} When Alexandrov spaces are glued along non-extremal subsets, it may still produce an Alexandrov space. In the following we glue two flat triangle planes, where $\measuredangle A_1B_1C_1+\measuredangle A_2B_2C_2=\pi$ and edge $\overline{B_1 C_1}$ is glued with edge $\overline{B_2 C_2}$. The glued space is also a triangle. When $\measuredangle A_1B_1C_1>\frac\pi2$, edge $\overline{B_1 C_1}$ is not an extremal subset in the triangle plane $\triangle A_1B_1C_1$. If $\measuredangle A_2B_2C_2+\measuredangle A_1B_1C_1>\pi$, then the glued space is not convex, which is not an Alexandrov space.

\begin{center}\includegraphics[scale=1]{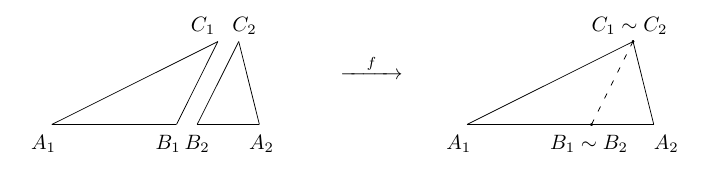}\end{center}

\begin{center} Figure 2\end{center}

\end{example}

\begin{example}\label{eg2} This is an example for self-gluing with $m_0=3$. Let $X$ be a triangle in a Euclidean plane. We identify points on each side via a reflection about the mid point, that is, $\overline{Ab}\sim\overline{bC}$, $\overline{Ac}\sim\overline{cB}$, $\overline{Ba}\sim\overline{aC}$ isometrically. The glued space $Y\in\Alex^2(0)$ is a tetrahedron. We see that $G_X^2=\overline{AB}\cup\overline{BC}\cup\overline{AC}\setminus\{A,B,C,a,b,c\}$ is open dense in $\partial X$ with $\dim_H(G_X^2)=1$ and $G_X^3=\{A, B, C\}$ consists of isolated points with $\dim_H(G_X^3)=0$.

\begin{center}\includegraphics[scale=1]{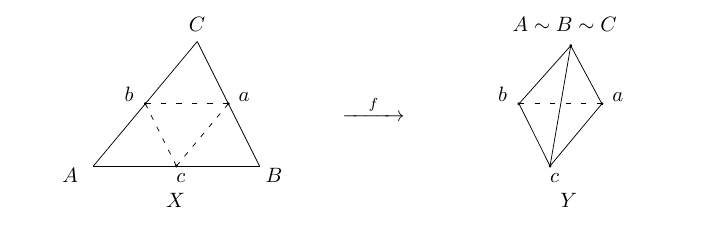}\end{center}

\begin{center} Figure 3\end{center}

\end{example}

\begin{example}\label{eg4}
  Let $Y=B_{\frac\pi{2}}(o,\Bbb S_1^n)$ be the semi $n$-sphere. By Theorem \ref{main.thm} (3), the maximum gluing number $m_0\le \frac{\vol{B_{\pi}(o,\Bbb S_1^n)}}{\vol Y}=2$, which states that any gluing with more than 2 points identified does not result an Alexandrov space. This is also verified by Theorem \ref{rel.max}, in the case $\kappa=1$, $\Sigma_p=\Bbb S_1^{n-1}$ and $R=\frac\pi{2}$.
\end{example}

\section{Preliminaries}

In this section we recall some important properties for Alexandrov spaces (c.f. \cite{BGP} and \cite{OS94}). Let $X\in\Alexnk$. By $X^{(m,\delta)}(\rho)$ we denote the collection of points which admit  $(m,\delta)$-strainers $\{(a_i,b_i)\}_{i=1}^m$ of size $\rho=\underset{1\le i\le m}\min\{|pa_i|, |pb_i|\}>0$. Clearly,
 $X^{(m,\delta)}=\scup{\rho>0}{}X^{(m,\delta)}(\rho)$ is open in $X$. Moreover, we have

\begin{thm}[{\cite{BGP}} 10.6] \label{bgp10.6}
  Let $X\in\Alexnk$. For $1\le m\le n$ and any $\delta>0$, $$\dim_H\left(X\setminus X^{(m,\delta)}\right)\le m-1.$$
\end{thm}

The local structures at the points in $X^{(n,\delta)}$ and $X^{(n-1,\delta)}$ satisfy the following three theorems. We let $X^\delta(\rho)=X^{(n,\delta)}(\rho)$ and $X^\delta=X^{(n,\delta)}$.


\begin{thm}[{\cite{BGP}} Theorem 9.4]\label{bgp9.4}
  Let $X\in \Alexnk$. If $p\in X^\delta(\rho)$, then the
  map $\psi:X\rightarrow \Bbb R^n$, $x\mapsto(|a_1x|, \cdots, |a_nx|)$ maps a small neighborhood $U$ of $p$ $\tau(\delta,\delta_1)$-almost isometrically
  onto a domain in $\Bbb R^n$, that is, $||xy|_X-|\psi(x)\psi(y)|_{\Bbb R^n}|<\tau(\delta,\delta_1)|xy|_X$ for any
  $x,y\in U$, where $\delta_1=\rho^{-1}\cdot\text{diam}(U)$. In particular, for $r<\delta\rho$, $\psi$ is a $\tau(\delta)$-almost isometry when restricting to $B_r(p)$.
\end{thm}

\begin{thm}[{\cite{BGP}} 12.8] \label{bgp12.8}
  Let $X\in\Alexnk$. For any $p\in X^{(n-1,\delta)}$, if $p\in \intl X$, then $p\in X^{\tau(\delta)}$.
\end{thm}

\begin{thm}[{\cite{BGP}} 12.9.1]\label{bgp12.9.1}
  Let $X\in\Alexnk$ and $p\in X^{(n-1,\delta)}$ with a strainer of size $\rho$. If $p\in\partial X$, then there is a neighborhood $U$ of $p$, which admits a
  $\tau(\delta,\delta_1)$-almost isometry (here $\delta_1=\rho^{-1}\cdot\diam(U)$) onto a closed cube in $\Bbb R^n$ such that $U\cap\partial X$ maps onto one of the hyperfaces of this cube.
\end{thm}

A consequence of Theorems \ref{bgp10.6} and \ref{bgp12.8} is that

\begin{cor} [{\cite{BGP}} 10.6.1] \label{bgp10.6.1} Let $X\in\Alexnk$.
For sufficiently small $\delta>0$, $\dim_H\left(\intl X\setminus X^\delta\right)\le n-2$.
\end{cor}

Let $\reg X=\scap{\delta>0}{}X^\delta$. By the Maximal volume rigidity theorem for Alexandrov spaces, for any $p\in\reg X$, the space of directions $\Sigma_p=\Bbb S_1^{n-1}$. 

\begin{thm}[{\cite{OS94}}]\label{os94}
  Let $X\in\Alexnk$. Then $\dim_H\left(X\setminus \reg X\right)\le n-1$.
\end{thm}

We also recall properties for spaces of directions in Alexandrov spaces.

\begin{thm}[{\cite{BGP}} 7.14]\label{bgp7.14}
  Let $p_i\in X\in\Alexnk$ and $p_i\to p$. Then $\dsp\liminf_{i\to\infty}\Sigma_{p_i}\ge\Sigma_p$, that is, for any convergent subsequence $\Sigma_{p_{i_k}}\overset{d_{GH}}\longrightarrow\Sigma$, there is a distance non-decreasing map $g\colon \Sigma_p\to\Sigma$.
\end{thm}

\begin{thm}[\cite{Pet98}]\label{pet98}
  Let $X\in \Alexnk$. Then for any $x,y\in\geod{pq}_X\setminus\{p,q\}$, $\Sigma_x$ is isometric to $\Sigma_y$.
\end{thm}

By Theorems \ref{bgp7.14} and \ref{pet98}, we get the following corollary.

\begin{cor}
  Let $X\in\Alexnk$. Then $\reg X$ is convex in $X$.
\end{cor}

We now consider the volume of small balls around almost regular points in Alexandrov spaces. The following are consequences of Theorem \ref{bgp9.4} and direct computations in Euclidean spaces.

\begin{lem}\label{vol.2ball} Let $X\in\Alexnk$.
  \begin{enumerate}
    \item For $p\in X^\delta(\rho)$ and $r<\delta\rho$,
      \begin{align*}
         (1+\tau(\delta))\cdot\text{vol}(B_r(p))&=\vol{B_r(o,\Bbb R^n)}  =\vol{\Bbb S_1^{n-1}}\int_0^r t^{n-1}\,dt  \\
         &=2r\cdot\vol{B_r(o,\Bbb R^{n-1})}\int_{0}^{\frac\pi2}\sin^n(t)dt.\end{align*}
    \item \footnote{This formula, together with Lemma \ref{delta.lip}, fixes an error occurred in \cite{LR10} Lemma 1.4}
        For $x_1, x_2\in X^\delta(\rho)$ with $|x_1x_2|\le 2r<\delta\rho/5$,
        \begin{align*}
          &(1+\tau(\delta))\cdot
          \vol{B_r(x_1)\cup B_r(x_2)}
          \\
          &\qquad=\vol{B_r(o,\Bbb R^n)} +2r\cdot\vol{B_r(o,\Bbb R^{n-1})}\int_{\theta}^{\frac\pi2}\sin^n(t)dt,
        \end{align*}
        where $\theta=\cos^{-1}\left(\frac{|x_1x_2|}{2r}\right)$.
  \end{enumerate}
\end{lem}


%
%
%
%

The following theorem is an immediate consequence of Theorem 10.2 in \cite{BGP} and the fact that the Hausdorff measure is a continuous function in the Hausdorff distance between Alexandrov spaces that have the same dimension and the same lower curvature bound. Here we give a proof without using the volume continuity.

\begin{thm}[Almost maximum volume]\label{abs.max}

  Let $A\in\Alex^n(1)$. If $\vol A\ge \vol{\Bbb S_1^n}-\epsilon$, then there exists a $\tau(\epsilon)$-Gromov-Hausdorff approximation $h:A\to \Bbb S_1^n$. In particular, if $p\in X\in\Alexnk$ and $\vol{\Sigma_p(X)}\ge\vol {\Bbb S_1^{n-1}}-\delta$, then $p\in X^{\tau(\delta)}$.
\end{thm}

\begin{proof}
  We inductively define the distance non-decreasing map $h=h_n:X\to\Bbb S_1^n$. The case for $n=1$ is trivial. Let $p\in X$, then $\Sigma_p\in\Alex^{n-1}(1)$. Assume that $h_{n-1}: \Sigma_p\to \Bbb S_1^{n-1}$ is defined and is distance non-decreasing, then $h_n=(h_{n-1},id)\circ \exp_p^{-1}$ is defined by the composition (c.f. \cite{BGP} 10.2):
  $$\begin{CD}
    X@>\exp_p^{-1}>>
    C_1^\pi(\Sigma_p)@>(h_{n-1},id)>>
    C_1^\pi(\Bbb S_1^{n-1})=\Bbb S_1^n,
  \end{CD}$$
  where $C_1^\pi(\Sigma_p)$ is the spherical suspension of $\Sigma_p$. Clearly $h_n$ is also distance non-increasing.

  Let
  $\Omega=\Bbb S_1^n\setminus h(X)$. We have
  $$\vol \Omega=\vol{\Bbb S_1^n}-\vol{h(X)}
    \le \vol{S_1^n}-\vol X
    < \epsilon.
  $$
  Let $B_r\subset \Bbb S_1^n$ be the metric ball which is not contained in $h(X)$, i.e., $B_r\subseteq \Omega$. Then
  $$\epsilon>\vol \Omega\ge\vol{B_r}=\vol{\Bbb S_1^{n-2}}\cdot\int_0^r\sin^{n-2}(t)\,dt.$$
  Thus $r<\tau(\epsilon)$ and $h$ is a $\tau(\epsilon)$-onto.

  We now show that $h$ is a $\tau(\epsilon)$-isometry. Let $p,x\in X$, $\tilde p=h(p)$ and $\tilde x=h(x)\in\Bbb S_1^n$. It's clear that $|\tilde p\tilde x|_{\Bbb S_1^n}\ge|px|_X$. Let $q$ be a point in $X$ such that $|p\,q|_X=\underset{t\in X}\sup\{|p\,t|_X\}=L$ and $\tilde q=h(q)\in\Bbb S_1^n$. Because
  $$\vol{\Bbb S_1^n} -\epsilon
     \le \vol X
     \le \vol{B_L(o,\Bbb S_1^n)},
  $$
  we have $L\ge \pi-\tau(\epsilon)$. By the inequalities
  \begin{align*}
    2\pi
    &\ge |\tilde p\tilde x|_{\Bbb S_1^n}
    +|\tilde p\tilde q|_{\Bbb S_1^n}
    +|\tilde x\tilde q|_{\Bbb S_1^n}
    \ge |\tilde p\tilde x|_{\Bbb S_1^n}
    +|pq|_X+|xq|_X
    \\
    &\ge |\tilde p\tilde x|_{\Bbb S_1^n}
    +|pq|_X+(|pq|_X-|px|_X)
    = |\tilde p\tilde x|_{\Bbb S_1^n}
    +2L-|px|_X,
  \end{align*}
  We get that $|\tilde p\tilde x|_{\Bbb S_1^n}-|px|_X \le2\pi-2L<\tau(\epsilon)$.
\end{proof}

\section{Gluing dimensions and isometry on interiors}

In this section, we prove Lemma \ref{thmC.int.iso}. If not explicitly stated, the assumptions in this section will be the same as in Theorem \ref{main.thm}. For a minimal geodesic $\geod{pq}_X$ in $X$, we let $\geodic{pq}_X=\geod{pq}_X\setminus\{p\}$, $\geodci{pq}_X=\geod{pq}_X\setminus\{q\}$ and $\geodii{pq}_X=\geod{pq}_X\setminus\{p,q\}$. We first prove some basic properties (Lemmas \ref{delta.img} -- \ref{delta.lip}) for the map $f$.


\begin{lem}\label{delta.img}
$f(X^\delta)\subseteq Y^{\tau(\delta)}$. In particular, $f(\reg X)\subseteq \reg Y$.
\end{lem}

\begin{proof}
  Let $x\in X^\delta(\rho)$ and $y=f(x)$. For $\epsilon\ll\delta\rho$, because $f$ is volume preserving and $f^{-1}(B_\epsilon(y))\supseteq B_\epsilon(x)$, we have the following volume comparison:
  \begin{align*}
    \vol{\Sigma_y}\cdot\int_0^\epsilon \snk^{n-1}(t)\,dt
    &\ge \vol{B_\epsilon(y)}
    \\
    &= \vol{f^{-1}(B_\epsilon(y))}
    \ge \vol{B_\epsilon(x)}
    \\
    &=(1-\tau(\delta))\cdot\vol{\Bbb S_1^{n-1}}\cdot\int_0^\epsilon t^{n-1}\,dt.
  \end{align*}
  Let $\epsilon\to 0$. We get $\vol{\Sigma_y}\ge(1-\tau(\delta))\vol{\Bbb S_1^{n-1}}$. Thus $y\in Y^{\tau(\delta)}$ by Theorem \ref{abs.max}.
\end{proof}

Recall that $G_Y=\{y\in Y: |f^{-1}(y)|\ge1.\}$ and $G_X=f^{-1}(G_Y)$. For simplicity, we write $x=f^{-1}(y)$ if $y\notin G_Y$. We aim to show that $G_X\subseteq\partial X$. Using volume comparisons, we can show $G_X\subseteq X\setminus X^\delta$, that is, for any $y\in f(X^\delta)$, there is a unique $x\in X$ such that $f(x)=y$.

\begin{lem}\label{delta.inj}
Let $d_0=\max\{\diam(X_\ell)\}$, $v_0=\min\{\vol{X_\ell}\}$. There exists a constant $c=c(n,\kappa, d_0, v_0)$ such that if $0<\delta<c$, then $X^\delta\cap\, G_X=\varnothing$. Consequently, $f(X^\delta)=Y\setminus f(X\setminus X^\delta)$ is open in $Y$ and for any $A\subseteq X$, $f(A\setminus X^\delta)=f(A)\setminus f(X^\delta)$.
\end{lem}

\begin{proof}
  We argue by contradiction. Assume that $x_1\neq x_2$, $f(x_1)=f(x_2)=y$ and $x_1\in X^\delta$. By Lemma \ref{delta.img}, $y\in Y^{\tau(\delta)}$. Suppose that $x_2\in X_\ell$ and $\diam(X_\ell)=D_\ell$. Let $\epsilon>0$ be small so that $B_\epsilon(x_1)\cap B_\epsilon(x_2)=\varnothing$. By Bishop-Gromov relative volume comparison for Alexandrov spaces (\cite{BBI}, \cite{LR10}), we have
  \begin{align*}
    1&=\frac{\vol{f^{-1}(B_\epsilon(y))}}{\vol{B_\epsilon(y)}}
    \ge \frac{\vol{B_\epsilon(x_1)}+\vol{B_\epsilon(x_2)}}{\vol{B_\epsilon(y)}}
    \\
    &\ge \frac{\vol{B_\epsilon(x_1)} +\vol{X_\ell}\cdot\frac{\int_0^\epsilon \snk^{n-1}(t)\,dt}{\int_0^{D_\ell} \snk^{n-1}(t)\,dt}} {\vol{B_\epsilon(y)}}
    \\
    &\ge \frac{(1-\tau(\delta))\cdot\vol{\Bbb S_1^{n-1}}\cdot \int_0^\epsilon t^{n-1}\,dt +v_0\cdot\frac{\int_0^\epsilon \snk^{n-1}(t)\,dt}{\int_0^{d_0} \snk^{n-1}(t)\,dt}} {(1+\tau(\delta))\cdot\vol{\Bbb S_1^{n-1}}\cdot \int_0^\epsilon t^{n-1}\,dt}.
  \end{align*}
  Let $\epsilon\to 0$. We get
  $$1\ge \frac{(1-\tau(\delta))\cdot\vol{\Bbb S_1^{n-1}}
    +\frac{v_0}{\int_0^{d_0} \snk^{n-1}(t)\,dt}}
    {(1+\tau(\delta))\cdot\vol{\Bbb S_1^{n-1}}}.
  $$
  This is a contradiction for $\delta>0$ sufficiently small.
\end{proof}

\begin{lem}\label{delta.lip} There are $\delta,\rho>0$ sufficiently small so that the following holds. For any $y_1, y_2\in f(X^\delta(\rho))$ that satisfies $|y_1y_2|_Y<\delta\rho/20$, we have
$$|f^{-1}(y_1) f^{-1}(y_2)|_X<(1+\tau(\delta))\cdot|y_1y_2|_Y.$$
\end{lem}

\begin{proof}
  Let $x_i=f^{-1}(y_i)$, $i=1,2$ and $\lambda=\frac{|x_1x_2|_X}{|y_1y_2|_Y}$.
  Take $r=\min\left\{\frac12|x_1x_2|_X, |y_1y_2|_Y\right\}$ and consider the balls $B_r(y_1)$ and $B_r(y_2)$. Clearly, $r\le |y_1y_2|_Y<\delta\rho/10$. By
  the volume formula Lemma \ref{vol.2ball} (2), we get
\begin{align*}
  &(1+\tau(\delta))\cdot\vol{B_r(y_1)\cup B_r(y_2)}
  \\
  &\qquad=\vol{B_r(o,\Bbb R^n)}
  +2r\cdot\vol{B_r(o,\Bbb R^{n-1})}
  \int_{\theta}^{\pi/2}\sin^n (t)\,dt
  \\
  &\qquad=2r\cdot\vol{B_r(o,\Bbb R^{n-1})}
  \int_{0}^{\pi/2}\sin^n (t)\,dt
  +2r\cdot\vol{B_r(o,\Bbb R^{n-1})}
  \int_{\theta}^{\pi/2}\sin^n (t)\,dt,
\end{align*}
where $\theta=\cos^{-1}\left(\frac{|y_1y_2|_Y}{2r}\right)$.
Note that $r\le \frac12|x_1x_2|_X$. Thus $B_r(x_1)\cap B_r(x_2)=\varnothing$. We have
\begin{align*}
  &(1+\tau(\delta))\cdot\vol{B_r(x_1)\cup B_r(x_2)}
  \\
  &\qquad=2\vol{B_r(o,\Bbb R^n)}
  =4r\cdot\vol{B_r(o,\Bbb R^{n-1})}
  \int_{0}^{\pi/2}\sin^n (t)\,dt.
\end{align*}
Because $f$ is 1-Lipschitz, we have $f^{-1}(B_r(y_1)
\cup B_r(y_2))\supseteq B_r
(x_1)\cup B_r(x_2)$. Together with that $f$ is volume preserving, we get
\begin{align}
  1&=\frac{\vol{f^{-1}(B_r(y_1)\cup B_r(y_2))}}
  {\vol{B_r(y_1)\cup B_r(y_2)}}
  \geq\frac{\vol{B_r(x_1)\cup B_r(x_2)}}
  {\vol{B_r(y_1)\cup B_r(y_2)}}
  \notag\\
  &=(1-\tau(\delta))\frac{2\int_{0}^{\pi/2}\sin^n (t)\,dt}
    {\int_{0}^{\pi/2}\sin^n (t)\,dt
    +\int_{\theta}^{\pi/2}\sin^n (t)\,dt}.
  \label{delta.lip.e1}
\end{align}
We claim that $\frac12|x_1x_2|_X\le |y_1y_2|_Y$. If this is not true, then $r=|y_1y_2|_Y$. In this case $\theta=\frac\pi3$, which will yield a contraction when $\delta$ is small. Now $r=\frac12|x_1x_2|_X$ and $\cos\theta=\frac{|y_1y_2|_Y}{|x_1x_2|_X}$. Inequality (\ref{delta.lip.e1}) implies that $0<\theta<\tau(\delta)$ and thus $\frac{|x_1x_2|_X}{|y_1y_2|_Y}=\frac1{\cos\theta}<1+\tau(\delta)$.
\end{proof}

The following lemma follows immediately from Lemmas \ref{delta.img} -- \ref{delta.lip}.

\begin{lem}[Almost isometry]\label{delta.almost.iso}\quad

  \begin{enumerate}
    \item If $\geod{pq}_Y\subset f(X^\delta)$, then $\gamma=f^{-1}(\geod{pq}_Y)$, parameterized by arc length, is a Lipschitz curve satisfying $$\lh{\gamma|_{[t_1,t_2]}}< (1+\tau(\delta))\cdot|\gamma(t_1)\gamma(t_2)|_X.$$
    \item $f|_{X^\delta}$ is $(1+\tau(\delta))$-bi-Lipschitz. In particular, if geodesic $\geod{f(a)f(b)}_Y\subset f(X^\delta)$, then
        \begin{equation}\label{int.iso.e1}
          1\le\frac{|ab|_X}{|f(a)f(b)|_Y}< 1+\tau(\delta).
        \end{equation}
    \item $f(X^\delta)\subseteq Y^{\tau(\delta)}$ is open and dense in $Y$.
  \end{enumerate}
\end{lem}

We want to show that the continuous extension $f|_{\intl{X_\ell}}$ of $f|_{X_\ell^\delta}$ is an isometry. 
Given a Lipschitz curve $\gamma\subset \intl{X_\ell}$, it's easy to construct a sequence of piecewise geodesics $\sigma_\epsilon\subset Y$ such that $\sigma_\epsilon\to f(\gamma)$ and $\lh{\sigma_\epsilon}\to \lh{f(\gamma)}$. The assertion will be proved if for each small $\delta>0$, $\sigma_\epsilon$ can be selected so that $\sigma_\epsilon\subset f(X_\ell^\delta)$ and $f^{-1}(\sigma_\epsilon)\to \gamma$.

The main difficulty in finding such approximation is that $f(X_\ell^\delta)$ may not be convex in $f(\intl X)$. The following lemma is a basic tool of dimension comparison (compare to Example \ref{eg.cube.edge}).

\begin{lem}[Dimension comparison]\label{map.dim}
  Let $\Omega_0\subseteq X\in\Alexnk$ be a subset and $p\in X$ be a fixed point. Let $\Omega\subseteq X$ be a subset such that $d(p,\Omega)>0$. If for each $x\in \Omega_0$ there exists a geodesic $\geod{px}_X$ such that $\Omega\cap\geod{px}_X\neq\varnothing$,
  then
  $$\dim_H(\Omega)\ge\dim_H(\Omega_0)-1.$$
\end{lem}

\begin{proof}
  Let $\Gamma=\Omega\times[0,\infty)$, equipped with the metric
  $$d((x_1,t_1),(x_2,t_2))=|x_1x_2|_X+|t_1-t_2|,$$
  where $x_i\in\Omega$, $t_i\in[0,\infty)$, $i=1,2$.
  Define a map $h: \Omega_0\to\Gamma$, $x\mapsto (\bar x,|px|_X)$, where $\geod{px}_X$ is selected such that $\geod{px}_X\cap\Omega\neq\varnothing$ and $\bar x\in \geod{px}_X\cap\Omega$ is selected arbitrarily. It suffices to show that the map $h$ is co-Lipschitz, that is, $h(B_r(x))\supseteq B_{cr}(h(x))\cap h(B_{cr}(x))$ for any $x\in\Omega_0$ and any $r>0$. It's sufficient to find a constant $c>0$ such that for any $x_1,x_2\in \Omega_0$,
  \begin{align}
    |h(x_1)h(x_2)|_{\Gamma}\ge c\cdot|x_1x_2|_X.
    \label{map.dim.e1}
  \end{align}
  By Fubini's theorem,
  $$\dim_H(\Omega)+1\ge\dim_H(\Gamma)\ge \dim_H(\Omega_0).$$

  We prove (\ref{map.dim.e1}) by triangle comparisons. If geodesics $\geod{px_1}_X$ and $\geod{px_2}_X$ are equivalent (i.e., one lies on the other), then
  $$\frac{|h(x_1)h(x_2)|_{\Gamma}}{|x_1x_2|_X}
    =\frac{|\bar x_1\bar x_2|_X+||px_1|_X-|px_2|_X|}{|x_1x_2|_X}
    \ge \frac{||px_1|_X-|px_2|_X|}{|x_1x_2|_X} =1.
  $$
  If geodesics $\geod{px_1}_X$ and $\geod{px_2}_X$ are not equivalent.
  Note that $|p\bar x_1|_X$, $|p\bar x_2|_X\ge d_X(p,\Omega)>0$. We have
  $$\frac{|h(x_1)h(x_2)|_{\Gamma}}{|x_1x_2|_X}
    =\frac{|\bar x_1\bar x_2|_X+||px_1|_X-|px_2|_X|}{|x_1x_2|_X}
    \ge \frac{|\bar x_1\bar x_2|_X}{|x_1x_2|_X} \ge c(\kappa,d_X(p,\Omega), \diam(\Omega_0))>0.
  $$
\end{proof}

Now we prove a key lemma.

\begin{lem}[Dimensions of boundary gluing]\label{glue.dim}
Assume $G_X\neq\varnothing$. Let ${a}\in G_X$. For any $\delta, r>0$, we have $\dim_H(B_r({a})\setminus X^\delta)\ge n-1$. Moreover, $G_X\subseteq\partial X$ and thus $f(\intl X)=Y\setminus f(\partial X)$ is open.
\end{lem}

\begin{proof}
Let ${a}\neq {b}\in G_X$ with $f({a})=f({b})=p\in G_Y$. Not losing generality, assume ${a}\in X_1$ and ${b}\in X_\ell$ ($\ell$ may equal to $1$). Because $f$ is 1-Lipschitz, it's sufficient to consider the Hausdorff dimension for $f(B_r({a})\setminus X^\delta)$.

  \medskip

  \begin{center}\includegraphics[scale=1]{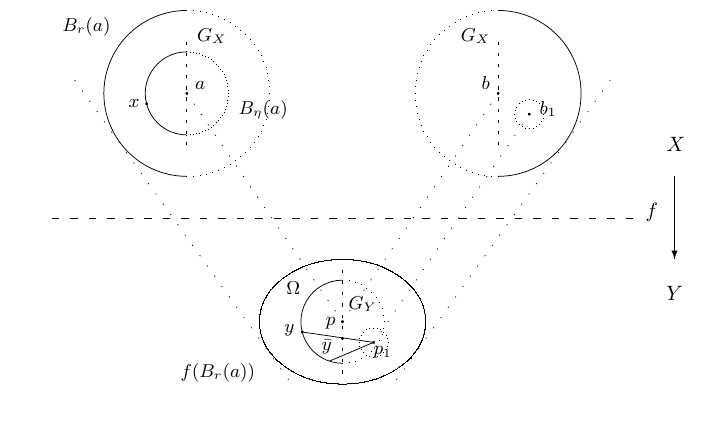} \end{center}

  \begin{center} Figure 4 \end{center}

  By Lemma \ref{delta.inj}, select $\delta>0$ small so that $f(X^\delta)\cap G_Y=\varnothing$, and thus $f(B_r({a})\setminus X^\delta)=f(B_r({a}))\setminus f(X^\delta)$. By Theorem \ref{os94}, for any $\eta>0$ small, there is ${b}_1\in \reg X_\ell$ with $|{b}{b}_1|_{X_\ell}<\eta$. By Lemma \ref{delta.img}, $p_1=f({b}_1)\in \reg Y$. Let $\Omega=f(B_{\eta}({a})\cap X_1^{\delta/2})$. It's clear that
  $$\dim_H(\Omega)=n.$$
  We first claim that for any $y\in \Omega$, there is a geodesic $\geod{yp_1}_Y$ such that $\geod{yp_1}_Y- f(X^\delta)\neq\varnothing$. If not so, then $\geod{yp_1}_Y\subset f(X^\delta)$. Let $x=f^{-1}(y)\in B_{\eta}({a})\cap X_1^{\delta/2}$. By the almost isometry of $f|_{X^\delta}$ (Lemma \ref{delta.almost.iso}), we get $|yp_1|_Y=(1-\tau(\delta))|x{b}_1|_X$. Consequently,
  \begin{align*}
    2\eta&\ge|x{a}|_X+|{b}{b}_1|_X \ge|yp|_Y+|pp_1|_Y
    \\
    &\ge |yp_1|_Y=(1-\tau(\delta))|x{b}_1|_X
    \ge (1-\tau(\delta))(|{a}{b}|_X-2\eta).
  \end{align*}
  This is a contradiction if $\delta$ and $\eta$ are sufficiently small.

  Take $\bar y\in\geod{yp_1}_Y\setminus f(X^\delta)$ which is closest to $y$ (see Figure 4). It's clear that $\bar y\neq y$. Moreover, $\bar y\notin f(X^\delta)$. This is because $f(X^{\delta})=Y\setminus f(X\setminus X^{\delta})$ is open in $Y$ (by Lemma \ref{delta.inj}) and thus $\geod{yp_1}_Y\setminus f(X^\delta)$ is closed. We claim that
  \begin{align}\bar y\in f(B_r({a})).\label{glue.dim.e0}\end{align}

  To prove the claim, we shall find a point $\bar x\in B_r({a})$ such that $\bar y=f(\bar x)$. By the construction, $\geodci{y\bar y}_Y\subset f(X^\delta)$. Let $y_i\in\geodci{y\bar y}_Y$ with $y_i\to \bar y$. By the almost isometry of $f$ (Lemma \ref{delta.almost.iso}), there are $x_i=f^{-1}(y_i)\in X^\delta$ such that
  \begin{equation}
    \label{glue.dim.e1}(1+\tau(\delta))|yy_i|_Y=|xx_i|_X.
  \end{equation}
  Passing to a subsequence, we let $\dsp\bar x=\lim_{i\to\infty} x_i$. Clearly, because $f$ is continuous, we have $f(\bar x)=\bar y$ and
  \begin{align}\label{glue.dim.e2}
    |x\bar x|_X=\lim_{i\to\infty}|xx_i|_X
    =(1+\tau(\delta))\lim_{i\to\infty}|yy_i|_Y=(1+\tau(\delta))|y\bar y|_Y.
  \end{align}
  Thus
  \begin{align*}
    |{a}\bar x|_X
    &\le|{a}x|_X+|x\bar x|_X\\
    &\le\eta+(1+\tau(\delta))|y\bar y|_Y
    \le\eta+(1+\tau(\delta))|yp_1|_Y\le (3+\tau(\delta))\eta.
  \end{align*}
  Choose $\eta>0$ small. We get $\bar x\in B_r({a})$.

  Now we let $\Omega_1$ be the collection of all of the above $\bar y$ selected for each $y\in \Omega$. Then $\Omega_1\subset f(B_r({a}))\setminus f(X^\delta)$. Note that $p_1\in \reg Y$. There is a small ball $B_\epsilon(p_1)\subset f(X^{\delta/2})$. Thus for any of the above selected $\bar y\notin f(X^\delta)$, we have $|\bar yp_1|_Y\ge \epsilon$. By the dimension comparison (Lemma \ref{map.dim}), we get
  \begin{align}
    \dim_H(\Omega_1)\ge\dim_H(\Omega)-1=n-1.\label{glue.dim.e3}
  \end{align}

  At last, by Corollary \ref{bgp10.6.1}, $\dim_H(\intl X\setminus X^\delta)\le n-2$, we conclude that $G_X\subseteq\partial X$.
\end{proof}

Now we use Lemmas \ref{map.dim} and \ref{glue.dim} to construct the desired perturbation of $f(\gamma)$. 

\begin{lem}\label{delta.adj}
  For any $p\in f(X^\delta)$ and $y\in Y$. If there is a convex neighborhood $U$ of $y$ in $Y$ such that
  $\geod{py}_Y\subset U\subseteq f(\intl{X_\ell})$, then for any $\epsilon>0$, there is $y'\in B_\epsilon(y)$ and a geodesic $\geod{py'}_Y$ such that $\geod{py'}_Y\subset f(X^{2\delta})$.
\end{lem}

\begin{proof}
  If the assertion is not true, then for any $y'\in B_\epsilon(y)$ and any geodesic $\geod{py'}_Y$,
  $\geod{py'}_Y\setminus f(X^{2\delta})\neq\varnothing$. Let $\Omega=\{\bar y\in \geod{py'}_Y\setminus  f(X^{2\delta}):y'\in B_\epsilon(y)\}$. Because $p\in f(X^\delta)$, there is a small ball such that $B_r(p)\subset f(X^{2\delta})$ and thus $d_Y(p,\Omega)\ge r$. By the dimension comparison lemma, we get
  \begin{equation}\dim_H(\Omega)\ge n-1.\label{delta.adj.e1}\end{equation}
  On the other hand, take $\epsilon>0$ small so that $B_\epsilon(q)\subset U$. Because $U$ is convex in $Y$, we have $\geod{py'}_Y\subset U\subseteq f(\intl{X_\ell})$ for all $y'\in B_\epsilon(y)$.
  Consequently, $\Omega\subset f(\intl{X_\ell})\setminus f(X^{2\delta})=f(\intl{X_\ell}\setminus X^{2\delta})$ by Lemma \ref{delta.inj}. Thus
  $$\dim_H(\Omega)\le \dim_H \left(f(\intl X\setminus X^{2\delta})\right)
    \le \dim_H\left(\intl X\setminus X^{2\delta}\right)\le n-2,$$
  which contradicts to (\ref{delta.adj.e1}).
\end{proof}

\begin{proof}[{\bf Proof of Lemma \ref{thmC.int.iso}}]
  Let $\gamma:[0,1]\to \intl X$ be a Lipschitz curve and $\sigma=f(\gamma)$. Clearly, $\sigma$ is also a Lipschitz curve since $f$ is 1-Lipschitz. It remains to show $\lh{\sigma}\ge \lh\gamma$. Note that by Lemma \ref{glue.dim}, $f(\intl X)$ is open. Then for each $y\in\sigma$, there is a convex neighborhood $U_y\subset f(\intl X)$. The existence of such convex neighborhood is referred to \cite{PP93} 4.3. Because $\sigma$ is compact in $Y$, there is a finite covering $\{U_{y_{2i}}\}_{i=0}^{N}$ of $\sigma$. Let $t_{2i}\ge 0$ so that $\sigma(t_{2i})=y_{2i}$. Not losing generality, we can assume that $0=t_0<t_2<\cdots<t_{2N}=1$, $\sigma(t_0)=y_0$, $\sigma(t_{2N})=y_{2N}$ and $\sigma\cap U_{y_{2i}}\cap U_{y_{2(i+1)}}\neq\varnothing$. Let $y_{2i+1}\in \sigma\cap U_{y_{2i}}\cap U_{y_{2(i+1)}}$, for $i=0,1,\cdots,N-1$. Then $0=t_0<t_1<t_2<\cdots<t_{2N}=1$ and we have
  $$\lh{\sigma}\ge\sum_{j=0}^{2N-1}|y_jy_{j+1}|_Y.$$

  Now we use Lemma \ref{delta.adj} to find a ``good" perturbation of $\cup\geod{y_jy_{j+1}}_Y$. First choose $y_0'\in f(X^{\delta/2^{2N}})\cap B_{\epsilon/2N}(y_0)\cap U_{y_0}$. By the convexity of $U_{y_0}$, we have $\geod{y_0'y_1}_Y\subset U_{y_0}\subset f(\intl X)$. By Lemma \ref{delta.adj}, there is $y_1'\in B_{\epsilon/2N}(y_1)\cap U_{y_0}\cap U_{y_2}$ such that $\geod{y_0'y_1'}_Y\subset f(X^{\delta/2^{2N-1}})$. By the convexity of $U_{y_2}$, we have $\geod{y_1'y_2}_Y\subset U_{y_2}\subset f(\intl X)$. Repeat the above process recursively for $j=1,2,\cdots, 2N$. We obtain a sequence $\{y_j'\}_{j=0}^{2N}$ with $y_j'\in B_{\epsilon/2N}(y_j)$ such that $\geod{y_j'y_{j+1}'}_Y\subset f(X^{\delta/2^{2N-(j+1)}})\subset f(X^\delta)$ for each $j$. Because $y_j'$ are $\epsilon/2N$-close to $y_j$, we have
  \begin{align*}
    \lh{\sigma}\ge\sum_{j=0}^{2N-1}|y_jy_{j+1}|_Y
    &\ge \sum_{j=0}^{2N-1}\left(|y'_jy'_{j+1}|_Y -\frac{2\epsilon}{2N}\right)
    =\sum_{j=0}^{2N-1}|y'_jy'_{j+1}|_Y-2\epsilon.
  \end{align*}
  Let $x_i=f^{-1}(y_i')$. By the almost isometry (Lemma \ref{delta.almost.iso}), we have $|y'_iy_{i+1}'|_Y=(1-\tau(\delta))|x_ix_{i+1}|_X$. Because $G_X\cap \intl X=\varnothing$, we have
  $\scup{i=0}{2N-1}\geod{x_ix_{i+1}}_X\to\gamma$ as $\epsilon\to 0$. Let $\epsilon\to 0$. We get
  \begin{align*}
    \lh{\sigma}
    \ge\liminf_{\epsilon\to 0}\sum_{i=0}^{2N-1}|y_i'y'_{i+1}|_Y
    \ge(1-\tau(\delta))\liminf_{\epsilon\to 0}\sum_{i=0}^{2N-1}|x_ix_{i+1}|_X
    \ge (1-\tau(\delta))\lh\gamma.
  \end{align*}
  The desired result is proved by taking $\delta\to 0$.
\end{proof}

\begin{cor}\label{int.map.quasi}
  Let $\gamma_i\subset \intl X$ be a sequence of geodesics. Suppose that $\gamma_i\to\gamma$. Then $f(\gamma)$ is a quasi-geodesic in $Y$. If $f|_\gamma$ is also injective, then $\lh\gamma=\lh{f(\gamma)}$. In particular, if $\geodii{pq}\subset \intl X$, then $f(\geod{pq})$ is a quasi-geodesic in $Y$.
\end{cor}

\begin{proof} Because $f|_{\intl X}$ is an isometry, $f(\gamma_i)$ are local geodesics. Thus $f(\gamma_i)$ are quasi-geodesics. Therefore $\dsp f(\gamma)=\lim_{i\to\infty}f(\gamma_i)$ is a quasi-geodesic. $\lh\gamma=\lh{f(\gamma)}$ follows from the fact that when quasi-geodesics converge, their lengths converge to the length of the limit quasi-geodesic.
\end{proof}

We would like to point out that the same technique does not work for a path $\gamma \subset \partial X_\ell$. See the following example.

\begin{example}\label{eg.bgy.conv}
Consider the gluing of a cylinder $X_1=\Bbb S^1\times[0,1]$ with two disks $X_2=X_3=\Bbb D^1$ (as caps), along their boundaries isometrically: $(\Bbb S^1\times\{0\})\sim \partial X_2$, $(\Bbb S^1\times\{1\})\sim \partial X_3$. We call the glued space $Y$ and let $f$ be the projection map. Fix a Lipschitz curve $\gamma\subset\Bbb S^1\times\{0\}\subset \partial X_1$. In the proof of Lemma \ref{thmC.int.iso}, $\sigma_\epsilon$ is selected to satisfy that $\dsp|\sigma_\epsilon(t_i)\gamma(t_i)|\ll \min_{j}\{|\gamma(t_j)\gamma(t_{j+1})|\}$. Because $f(X_1)$ is not convex in $Y$, such $\sigma_\epsilon $ in $Y$ would ``mainly'' stay in $f(X_2)$ and $f^{-1}(\sigma_\epsilon)\nsubseteq \partial X_1$. In fact, $f^{-1}(\sigma_\epsilon)$ will always converge (if they converge) to a curve in $\partial X_2$.

If $\sigma_\epsilon(t_i)=\gamma(t_i)$, by the same reason, the pre-image of the tangent vectors of $\sigma_\epsilon$ are in the tangent cones of points in $X_2$. Thus we can not see a clear connection between $\lh{f(\gamma)}$ and $\lh{\gamma}$ if $\gamma\subset \partial X_1$. See Remark \ref{rmk.df}.
\end{example}

\section{Length preserving along boundaries}

This section is dedicated to show that $f$ preserves the lengths of paths for all curves (Lemma \ref{thmC.L.pres}). If not explicitly stated, the assumptions in this section will be the same as in Theorem \ref{main.thm}.

We first construct a 1-Lipschitz map over the spaces of directions. For any $p\in Y$, we let
$$T_{f^{-1}(p)}(X)=\dsju{a\in f^{-1}(p)}{}T_a(X).$$
For each $a\in X_\ell$, we consider the maps $f_{\lambda_i,{a}}\colon(X_\ell,{a},\lambda_id_X) \to(Y,f({a}),\lambda_id_Y)$, where $\lambda_i\to\infty$. By Arzel\`a-Ascoli theorem, passing to a subsequence, $f_{\lambda_i,{a}}$ converges to a 1-Lipschitz map $df_a\colon T_{a}(X_\ell)\to T_{f({a})}(Y)$. We define $df_{f^{-1}(y)}\colon T_{f^{-1}(p)}\to T_p$ as the following: $df_{f^{-1}(p)}(v)= df_a(v)$ if $v\in T_a(X)$. Clearly, $df_{f^{-1}(p)}$ is 1-Lipschitz onto and volume preserving.

Because $f|_{\intl X}$ is an isometry and $f(\intl{X_{\ell_1}})\cap f(\intl{X_{\ell_2}})=\varnothing$ for $\ell_1\neq \ell_2$, we have that
\begin{align}
  &df_{a}|_{\intl{T_{a}(X)}} \text{ is an isometry;}\label{df.cp1}
  \\
  &\text{$df_{a}(\intl{T_{a}(X)})\cap df_b(\intl{T_{b}(X)})=\varnothing$ if ${a}\neq b$ and $f({a})=f(b)$.}\label{df.cp2}
\end{align}
Let $O_a$ denote the vertex of the tangent cone $T_a$. For each vector $v\in T_a(X)$, let $v_i\in \intl{T_a(X)}$ and $s_i\in\geod{O_av_i}$ such that $v_i\to v$ and $|O_a\,s_i|=\frac1i|vv_i|$. We have  $\geod{s_iv_i}\subset \intl{T_a(X)}$ and $\geod{s_iv_i}\to\geod{O_av}$ as $i\to\infty$. By Corollary \ref{int.map.quasi}, $\dsp df_a(\geod{O_av})=\lim_{i\to\infty}df_a(\geod{s_iv_i})$ is a quasi-geodesic in $T_p(Y)$. Moreover, $df_a|_{\geod{O_av}}$ is injective, because $df_a(\geod{s_iv_i})$ are interior of rays in $T_p(Y)$. Using the cone structure again, we get that the quasi-geodesic $\dsp df_a(\geod{O_av})$ coincides with the geodesic $\geod{O_p\, df_a(v)}$. Therefore by Corollary \ref{int.map.quasi},
\begin{align}
  |v|=\lh{\geod{O_av}}=\lh{df_a(\geod{O_av})} =\lh{\geod{O_p\, df_a(v)}}=|df_a(v)|.\label{df.cp3}
\end{align}
Let $$\Sigma_{f^{-1}(p)}(X)=\dsju{a\in f^{-1}(p)}{}\Sigma_a(X).$$
The restricted map (still called $df_{f^{-1}(p)}$)
\begin{align}
  df_{f^{-1}(p)}|_{\Sigma_{f^{-1}(p)}(X)}\colon \Sigma_{f^{-1}(p)}(X)\to\Sigma_p(Y)
  \label{df.sp0}
\end{align}
is 1-Lipschitz onto and inherits the properties (\ref{df.cp1}) and (\ref{df.cp2}), that is,
\begin{align}
  &df_{a}|_{\intl{\Sigma_{a}(X)}} \text{ is an isometry;}\label{df.sp1}
  \\
  &\text{$df_{a}(\intl{\Sigma_{a}(X)})\cap df_b(\intl{\Sigma_{b}(X)})=\varnothing$ if ${a}\neq b$ and $f({a})=f(b)$.}\label{df.sp2}
\end{align}
Consequently, for distinguished $a_j\in X$, $j=1,\dots,m$, that satisfy $f(a_j)=p$, we have
\begin{align}
  \sum_{j=1}^m\vol{\Sigma_{a_j}}\le\vol{\Sigma_p}.\label{df.vol0}
\end{align}
By (\ref{df.sp0}), the volume equality
\begin{align}
  \sum_{a\in f^{-1}(p)}\vol{\Sigma_{a}} =\vol{\Sigma_{f^{-1}(p)}}=\vol{\Sigma_p}\label{df.vol}
\end{align}
holds if one can show that $f^{-1}(p)$ has finite cardinality.

%
%


\begin{proof}[{\bf Proof of Theorem \ref{main.thm} (3)}]
  Let $p\in Y$ and $a_j\in X_{\ell_j}$ that satisfies $f(a_j)=p$. Let $D_{\ell_j}=\diam(X_{\ell_j})$.
  For each $j$,
  $$v_0\le\vol {X_{\ell_j}}\le \vol{\Sigma_{a_j}}\cdot\int_0^{D_{\ell_j}}\snk^{n-1}(t)\,dt \le\vol{\Sigma_{a_j}}\cdot\int_0^{d_0}\snk^{n-1}(t)\,dt.$$
  Summing up for $j=1,2,\cdots,m$, we get
  \begin{align*}
    m\cdot v_0
    &\le \sum_{j=1}^m\vol{\Sigma_{a_j}}\cdot\int_0^{d_0}\snk^{n-1}(t)\,dt.
  \end{align*}
  By the volume inequality (\ref{df.vol0}), we have
  $$\sum_{j=1}^m\vol{\Sigma_{a_j}}\le\vol{\Sigma_p}\le \vol{\Bbb S_1^{n-1}}.$$
  Thus
  $$m\cdot v_0\le \vol{\Bbb S_1^{n-1}}\cdot\int_0^{d_0}\snk^{n-1}(t)\,dt
    =\vol{B_{d_0}(o,\Bbb S_{\kappa}^n)}.
  $$
\end{proof}

\begin{remark}\label{rmk.df}
  Lemma \ref{thmC.L.pres} does not follow immediately from (\ref{df.cp3}). One of the issues is that, for a curve $\gamma\subset\partial X_\ell$, the construction of $df$ does not guarantee that $\dsp df_{\gamma(t)}\left(\gamma^+(t)\right)=\frac{d}{dt}\,f(\gamma(t))$. This is because $df_{\gamma(t)}\left(\gamma^+(t)\right)$ is defined only by the intrinsic geometry of $X_\ell$ that contains $\gamma$, that is, $f(X_\ell)$. However, $\dsp\frac{d}{dt}\,f(\gamma(t))$ is determined by all $f(X_\ell)$ that contain $f(\gamma)$ (see the non-isometric gluing Example \ref{eg.non.iso}). Let $\sigma(t)=f(\gamma(t))\subset Y$. By the definition, the tangent vector $\dsp\sigma^+(t)=\lim_{s\to 0}\drn_{\geod{\sigma(t)\sigma(t+s)}_Y}$ if it exists. However, the $df_{\gamma(t)}$-pre-images of $\drn_{\geod{\sigma(t)\sigma(t+s)}_Y}$ may not stay in $T_{\gamma(t)}(X_\ell)$ (see Example \ref{eg.bgy.conv}, $\gamma\subset\partial X_1$). On the other hand, so far we only know that $f(\geod{\gamma(t)\gamma(t+s)}_X)$ are quasi-geodesics connecting $\sigma(t)$ and $\sigma(t+s)$. In general, this is not sufficient to imply that their tangent vectors converge to $\dsp\frac{d}{dt}\,f(\gamma(t))$. In our case, one should try to approximate $\dsp\frac{d}{dt}\,f(\gamma(t))$ by vectors $v_i$ whose $df_{\gamma(t)}$-pre-images are vectors approximating $\gamma^+(t)$. As shown in Example \ref{eg.bgy.conv}, taking $v_i=\drn_{\geod{\sigma(t)\sigma(t+s)}_Y}$ would not work even if the gluing is by isometry. Lemmas \ref{delta.flat} and \ref{bdy.local.iso} are used to solve this issue.

 %
%
%
%
%
%
%
%
%
%
%
%
\end{remark}

\begin{lem}\label{delta.flat}
  For any $\delta>0$ small and ${a}\in\partial X_\ell$, $p=f({a})$, there is a small ball $B_r({a})$ in $X_\ell$ such that for any $x\in B_r({a})$ and $y=f(x)$,
  $$\left|df_x(\drn_{x}^{{a}})\drn_y^p\right|_{\Sigma_y}<\delta.$$
\end{lem}

It can happen that in any small neighborhood of ${a}$, there exits an $x$ such that $df_x(\drn_{x}^{{a}})\neq \drn_y^p$ (see Example \ref{eg.bgy.conv}, ${a}$ and $x$ are points in $\Bbb S^1\times\{0\}$).

\begin{proof}
  It's sufficient to show that for any $x_i\to a$ and $y_i=f(x_i)$,
  $$\left|df_{x_i}(\drn_{x_i}^{{a}})\drn_{y_i}^p\right|_{\Sigma_{y_i}} <\tau(|{a}x_i|).
  $$
  Take $s_i\in B_{{a}}(|py_i|^3)\cap\intl{X_\ell}$. We first show that $\left|\drn_{x_i}^{s_i}\drn_{x_i}^{a}\right|_{\Sigma_y}<\tau(|{a}x_i|)$. Consider the rescaling $\left(X_\ell, {a}, \frac{1}{|{a}x_i|}d\right)$, which Gromov-Hausdorff converge to the tangent cone $C(\Sigma_{{a}})$. Identify the points in $X_\ell$ and the corresponding points in its rescaling. Passing to a subsequence, let $O_{{a}}, \tilde x\in C(\Sigma_{{a}})$ be the limit points of ${a}$ and $x_i$ respectively. Clearly, $s_i\to O_{{a}}$.  Because the geodesic $\geod{O_{{a}}\tilde x}$ can be extended forward from $\tilde x$, we have that geodesics $\geod{s_ix_i}\to\geod{O_{{a}}\tilde x}$. Furthermore, there are $z_i\in X_\ell$ ($z_i$, after the rescaling, converge to a point on the extension of $\geod{O_{{a}}x_i}$), such that $\left|\drn_{x_i}^{s_i}\drn_{x_i}^{z_i}\right| \ge\tang{\kappa}{x_i}{s_i}{z_i}\ge\pi-\tau(|{a}x_i|)$ and $\left|\drn_{x_i}^{{a}}\drn_{x_i}^{z_i}\right| \ge\tang{\kappa}{{a}}{s_i}{z_i}\ge\pi-\tau(|{a}x_i|)$. Therefore, $\left|\drn_{x_i}^{s_i}\drn_{x_i}^{{a}}\right|<\tau(|{a}x_i|)$. Because $df_{f^{-1}(y_i)}$ is a 1-Lipschitz map, we have that
  \begin{equation}
    \left|df_{x_i}(\drn_{x_i}^{s_i})df_{x_i}(\drn_{x_i}^{{a}})\right|<\tau(|{a}x_i|).
    \label{delta.flat.e1}
  \end{equation}

  Now consider the rescaling $\left(Y, p, \frac{1}{|py_i|}d\right)$. Beacuse $|ax_i|\ge|py_i|\to 0$, the above sequence Gromov-Hausdorff converges to the tangent cone $C(\Sigma_p)$. Let $O_p, \tilde y\in C(\Sigma_p)$ be the limit points of $p$ and $y_i$ respectively. Clearly, $f(s_i)\to O_p$. By Corollary \ref{int.map.quasi}, $f(\geod{x_is_i})$ are quasi-geodesics. Thus they converge to a quasi-geodesic from $O_p$ to $\tilde y$ in $C(\Sigma_p)$, which coincides with $\geod{O_p\tilde y}$. Note that by the definition of $df_{x_i}$, $\drn_{x_i}^{s_i}$ is the tangent vector of $f(\geod{x_is_i})$ at $x_i$. By a similar argument as the above, we get that
  \begin{equation}
    \left|df_{x_i}(\drn_{x_i}^{s_i})\drn_{y_i}^{a})\right|<\tau(|py_i|)<\tau(|{a}x_i|).
    \label{delta.flat.e2}
  \end{equation}
  The desired estimate is obtained by combining (\ref{delta.flat.e1}) and (\ref{delta.flat.e2}).
\end{proof}

\begin{lem}\label{bdy.local.iso}
  Let $\delta>0$ be small. Then for any ${a}\in \partial X_\ell$, there is $r=r({a})>0$ such that for any ${b}\in B_{r}({a})$, if $f(b)\neq f(a)$, then
  $$1\ge\frac{|f({a})f({b})|_Y}{|{a}{b}|_{X_\ell}}\ge 1-\delta.$$
\end{lem}

\begin{proof}
  Let $B_{r_0}({a})\subset X_\ell$ be the neighborhood of ${a}$ chosen as in Lemma \ref{delta.flat} and take $r=r_0/10$.
  Let $p=f({a})$. Given ${b}\in B_{r}({a})$, and $q=f({b})\neq p$. It's sufficient to find a path $\gamma$ from ${b}$ to ${a}$ in $X_\ell$ such that $(1-\delta)\lh{\gamma}\le|pq|_Y$. Let $\mathcal I\subseteq [0,\,|pq|_Y]$ be the collection of $\omega_0\in [0,\,|pq|_Y]$ such that for each $\omega\in[\omega_0,|pq|_Y]$, there is a Lipschitz curve $\gamma:[0,T]\to X_\ell$ such that $\gamma(0)={b}$, $|p\,f(\gamma(T))|_Y=\omega$ and $(1-\delta)\lh{\gamma}\le|pq|_Y-\omega$. We claim that $\mathcal I=[0,\,|pq|_Y]$. Clearly, $\mathcal I\neq\varnothing$ since $|pq|_Y\in\mathcal I$ (take $\gamma\equiv\{b\}$ and $f(\gamma(T))=q$). $\mathcal I$ is also closed. Thus it suffices to show that $\mathcal I$ is open.

  Suppose $\omega_0\in\mathcal I$ and $\omega_0\neq 0$. There is a Lipschitz path $\gamma\subset X_\ell$ such that $\gamma(0)={b}$, $\gamma(T)={x}\in X_\ell$, $f(\gamma(T))=f({x})={y}\in Y$, $\omega_0=|p{y}|_Y>0$, and
  \begin{equation}
    |pq|_Y-|p{y}|_Y\ge(1-\delta)\lh{\gamma}.
    \label{bdy.local.iso.e0}
  \end{equation}
  \begin{center}\includegraphics[scale=1]{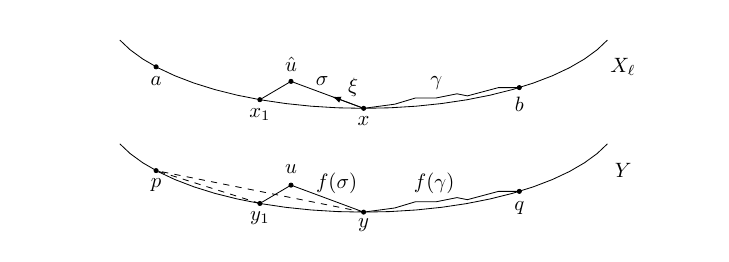}\end{center}

  \begin{center} Figure 5 \end{center}

  We want to find an extension of $\gamma$ toward ${a}$ which also satisfies (\ref{bdy.local.iso.e0}). If $\drn_{{x}}^{{a}}\,\notin\intl{\Sigma_{{x}}}$, we take a little perturbation $\xi\in\intl{\Sigma_{{x}}}$ with
  \begin{equation}
    |\xi \drn_{{x}}^{{a}}|_{\Sigma_{{x}}}<\delta.
    \label{bdy.local.iso.e01}
  \end{equation}
  Take a quasi-geodesic $\sigma\colon[0,\epsilon]\to X_\ell$ which satisfies $\sigma(0)={x}$, $\sigma^+(0)=\xi$ and $\sigma((0,\epsilon])\subset \intl{X_\ell}$. Let $\hat u=\sigma(\epsilon)\in \intl{X_\ell}$, $u=f(\hat u)\in f(\intl{X_\ell})$. Join $p$ and $u$ by a geodesic $\geod{pu}_Y$ in $Y$. Let ${y}_1\in\geod{pu}_Y\cap f(\partial X_\ell)$ which is closest to $u$. Because $f^{-1}(u)=\hat u\in \intl{X_\ell}$, there is ${x}_1\in\partial X_\ell$ and a geodesic $\geodci{\hat u{x}_1}_{X_\ell}\subset\intl{X_\ell}$ such that $\geod{u{y}_1}_Y=f(\geod{\hat u{x}_1}_{X_\ell})$. Consequently, we have
  ${y}_1=f({x}_1)$, and
  $$|\hat u{x}_1|_{X_\ell}=\lh{\geod{\hat u{x}_1}_{X_\ell}}=\lh{\geod{u{y}_1}_Y}
  =|u{y}_1|_Y.$$
  We claim that
  \begin{align}
    |p{y}|_Y-|p{y}_1|_Y\ge (1-\delta)(\epsilon+|u{y}_1|_Y)
    =(1-\delta)(\lh{\sigma}+|\hat u{x}_1|_{X_\ell}).
    \label{bdy.local.iso.e1}
  \end{align}
  This immediately implies that $|p{y}_1|_Y<|p{y}|_Y$. Summing up (\ref{bdy.local.iso.e1}) and (\ref{bdy.local.iso.e0}), we get
  \begin{align}
    |pq|_Y-|p{y}_1|_Y
    &\ge(1-\delta)(\lh{\gamma}+\lh{\sigma}+|\hat u{x}_1|_{X_\ell})
    \label{bdy.local.iso.e2}
    \\
    &=(1-\delta)\lh{\gamma\cup\sigma\cup\geod{\hat u{x}_1}_{X_\ell}},
    \notag
  \end{align}
  where $\gamma\cup\sigma\cup\geod{\hat u{x}_1}_{X_\ell}$ is a Lipschitz path connecting points in the following order: ${b}\to {x}\to \hat u\to{x}_1$. In fact, the above estimate is also valid if we replace ${y}_1$ by any point on $\sigma\cup\geod{\hat u{x}_1}_{X_\ell}$. Thus $\mathcal I$ is open.

  To see (\ref{bdy.local.iso.e1}), consider the triangle $\triangle p{y}u\subset Y$ which consists of $\geod{p{y}}_Y$, $\geod{pu}_Y$ and quasi-geodesic $f(\sigma)$ (by Corollary \ref{int.map.quasi}). Note that for $\delta>0$ small and for any path $\gamma$ satisfying (\ref{bdy.local.iso.e0}), we have
  \begin{align}
    |{a}{x}|_{X_\ell}
    &\le |{b}{x}|_{X_\ell}+|{a}{b}|_{X_\ell}
    \le \lh{\gamma}+r
    \label{bdy.local.iso.e001}
    \\
    &\le 2(|pq|_Y-|p{y}|_Y)+r
    \le 2|{b}{a}|_{X_\ell}+r\le 3r.
    \notag
  \end{align}
  Thus ${x}\in B_{4r}({a})\subset B_{r_0}({a})$. It's clear that $f(\sigma)^+(0)=df_s(\xi)$. By Lemma \ref{delta.flat},
  \begin{align}
    \left|f(\sigma)^+(0)\drn_{y}^p\right|_{\Sigma_{y}}
    &\le \left|f(\sigma)^+(0)\, df_{{x}}(\drn_{{x}}^{{a}})\right|_{\Sigma_{y}}
    +\left|df_{{x}}(\drn_{{x}}^{{a}})\drn_{y}^p\right|_{\Sigma_{y}}
    \label{bdy.local.iso.e01}
    \\
    &\le \left|\xi\, \drn_{{x}}^{{a}}\right|_{\Sigma_{{x}}}
    +\left|df_{{x}}(\drn_{{x}}^{{a}})\drn_{y}^p\right|_{\Sigma_{y}}
    <2\delta.
    \notag
  \end{align}
  Note that $\lh{f(\sigma)}=\lh{\sigma}=\epsilon$. By the first variation formula,
  $$|p{y}_1|_Y+|{y}_1u|_Y=|pu|_Y\le |p{y}|_Y-\cos(2\delta)\epsilon+\tau(\epsilon)\epsilon.$$
  Take $\epsilon>0$ small so that $\tau(\epsilon)<\frac12\delta$. Then for $\delta>0$ small, we obtain (\ref{bdy.local.iso.e1}):
  \begin{align*}
    |p{y}|_Y-|p{y}_1|_Y
    &\ge |u{y}_1|_Y+\cos(2\delta)\epsilon-\frac12\delta\epsilon
    \\
    &\ge |u{y}_1|_Y+(1-\delta)\epsilon
    \ge (1-\delta)(|u{y}_1|_Y+\epsilon).
  \end{align*}
Now we get that $0\in\mathcal I$. Let $\gamma:[0,T]\to X_\ell$ be a Lipschitz path which satisfies
$$f(0)=b , \quad f(\gamma(T))=p
  \text{\quad and \quad}
  |pq|_Y\ge(1-\delta)\lh{\gamma} .$$
It remains to show that $\gamma(T)={a}$. Since $f^{-1}(p)$ has finite cardinality (Theorem \ref{main.thm} (3)), we can take $r_0$ small enough so that $f^{-1}(p)\cap B_{r_0}({a})=\{{a}\}$. Thus it suffices to show that $\gamma(T)\in B_{r_0}({a})$. By a similar estimation as (\ref{bdy.local.iso.e001}), we get that
\begin{align*}
    |{a}\gamma(T)|_{X_\ell}
    &\le |{a}{b}|_{X_\ell}+|{b}\,\gamma(T)|_{X_\ell}+
    \le |{a}{b}|_{X_\ell}+\lh{\gamma}
    \\
    &\le |{a}{b}|_{X_\ell}+2|{a}{b}|_{X_\ell}
    \le 3r.
\end{align*}
Thus $\gamma(T)\in B_{4r}({a})\subset B_{r_0}({a})$.
\end{proof}

\begin{proof}[{\bf Proof of the path isometry (Lemma \ref{thmC.L.pres})}]

Let $\gamma\colon\mathcal I\to X_\ell$ be a curve. If $\gamma$ is rectifiable, we assume $\mathcal I=[0,1]$. For a non-rectifiable curve, it suffices to prove for the cases $\mathcal I=[0,1]$ and $\mathcal I=[0,\infty)$. Let $\rho, T>0$, where $T\in\mathcal I$. For each $a\in\gamma|_{[0,T]}$, there is a ball $B_r(a)$, $r<\rho$, that is either contained in $\intl{X_\ell}$ or satisfies Lemma \ref{bdy.local.iso}. The open intervals $\left(\gamma^{-1}(a)-\epsilon_a,\gamma^{-1}(a)+\epsilon_a\right)$, contained in the pre-images $\gamma^{-1}(B_r(a)\cap\gamma|_{[0,T]})$, form an open cover of the interval $[0,T]$. Thus there is a finite subcover of $[0,T]$, which corresponds to a finite subcover $\{B(a_i),\, i=0,\dots,N\}$ of $\gamma|_{[0,T]}$.  Not losing generality, assume that $a_i$ are points on $\gamma$ in the same order. Let $x_{2i}=a_i$ and $t_{2i}\in\mathcal I$ so that $\gamma(t_{2i})=x_{2i}$. Assume that $0=t_0<t_2<\cdots<t_{2N}=T$, where $\gamma(0)=x_0$, $\gamma(T)=x_{2N}$ and $\gamma|_{[0,T]}\cap B(x_{2i})\cap B(x_{2i+2})\neq\varnothing$. For $i=0,1,\cdots,N-1$, take $t_{2i+1}\in(t_{2i},t_{2i+2})$ such that $x_{2i+1}=\gamma(t_{2i+1})\in \gamma|_{[0,T]}\cap B(x_{2i})\cap B(x_{2i+2})$. Then $\cup\geod{x_i,x_{i+1}}$ is a polygonal approximation of $\gamma|_{[0,T]}$. Let $y_i=f(x_i)$. By Lemmas \ref{thmC.int.iso} and \ref{bdy.local.iso}, we have
 \begin{align*}
   \lh{f(\gamma)}&\ge\sum_{i=0}^{2N}|y_iy_{i+1}|_Y
   \ge (1-\delta)\sum_{i=0}^{2N}|x_ix_{i+1}|_{X_\ell}.
 \end{align*}

 If $\gamma$ is rectifiable, we take $T=1$. For any $\epsilon>0$, $\rho$ can be chosen small so that $\sum_{i=0}^{2N}|x_ix_{i+1}|_{X_\ell}\ge\lh{\gamma}-\epsilon$.  Let $\epsilon,\delta\to 0$, we get $\lh{f(\gamma)}\ge \lh{\gamma}$. If $\lh{\gamma|_{[0,T]}}=\infty$ for some $T>0$, then for such $T$ and any $L>0$, $\rho$ can be chosen small so that $\sum_{i=0}^{2N}|x_ix_{i+1}|_{X_\ell}>L$. Therefore, $f(\gamma)$ is also non-rectifiable. Suppose that $\gamma$ is non-rectifiable but $\lh{\gamma|_{[0,T]}}<\infty$ for every $T>0$. Then for any $L>0$, there is $T>0$ such that $\lh{\gamma|_{[0,T]}}>L$. For every such $T$, there is $\rho>0$ such that $\sum_{i=0}^{2N}|x_ix_{i+1}|_{X_\ell}>\lh{\gamma|_{[0,T]}}-1>L-1$. Thus $\lh{f(\gamma)}=\infty$.
\end{proof}

\begin{proof}[{\bf Proof of Theorem \ref{main.thm} (4)}]

Because $G_X\subseteq\partial X$ and $f$ is 1-Lipschitz, it suffices to show
\begin{align}
  \dim_H(B_r(f({a}))\cap G_Y)\ge n-1\label{pf.main.thm.e1}
\end{align}
and
\begin{align}
  \dim_H\left(\scup{m=3}{m_0}G_X^m\right)
  =\dim_H\left(\scup{m=3}{m_0}G_Y^m\right)
  \le n-2.\label{pf.main.thm.e2}
\end{align}

Estimate (\ref{pf.main.thm.e1}) is proved along the same line as the proof of Lemma \ref{glue.dim}. It suffices to show that assertion (\ref{glue.dim.e0}) can be strengthened to be $\bar y\in f(B_r({a}))\cap G_Y$. Then the set $\Omega_1\subset f(B_r({a}))\cap G_Y$ and (\ref{pf.main.thm.e1}) follows from (\ref{glue.dim.e3}). Because $p_1\in \reg Y$ and by Theorems \ref{bgp7.14} and \ref{pet98}, we see that $\geodic{yp_1}_Y\subset \reg Y$. In particular, $\bar y\in \reg Y$. By the volume equality (\ref{df.vol}), if $\bar y\notin G_Y$, then
$$\vol{\Sigma_{\bar x}}=\vol{\Sigma_{\bar y}}=\vol{\Bbb S_1^{n-1}}.$$
Therefore $\bar x\in \reg X$. This contradicts to the selection that $\bar y\notin f(X^\delta)$.

Now we prove (\ref{pf.main.thm.e2}). We claim that there is $\delta>0$ small so that for any $3\le m\le m_0$, if $y\in G_Y^m$, then $f^{-1}(y)\setminus X^{(n-1,\delta)}\neq\varnothing$. Thus by Theorem \ref{bgp10.6}, we get
$$\dim_H\left(\scup{m=3}{m_0}G_Y^m\right)
  \le \dim_H\left(X\setminus X^{(n-1,\delta)}\right)\le n-2.
$$

If the claim is not true, then $f^{-1}(y)=\{x_1,x_2,\cdots,x_m\}\subset X^{(n-1,\delta)}$. By Theorem \ref{bgp9.4}, $x_j\in X^{\tau(\delta)}$ or $x_j\in\partial X$. In either case, by Theorems \ref{bgp12.8} and \ref{bgp12.9.1}, we have $\vol{\Sigma_{x_j}}\ge\frac{1}{2}\vol{\Bbb S_1^{n-1}}-\tau(\delta)$. Thus
\begin{align*}
  \vol{\Bbb S_1^{n-1}}
  &\ge \vol{\Sigma_y}
  =\sum_{j=1}^m\vol{\Sigma_{x_j}}
  \\
  &\ge \sum_{j=1}^m\left(\frac{1}{2}\vol{\Bbb S_1^{n-1}}-\tau(\delta)\right)
  \ge \frac{m}{2}\vol{\Bbb S_1^{n-1}}-m\tau(\delta).
\end{align*}
This is a contraction for $m\ge 3$ and $\delta>0$ small.

It remains to show that $\dim_H\left(G_X^m\right)
  \le\dim_H\left(G_Y^m\right)$ for each $m\le m_0<\infty$. It suffices to show that $f$ is in a local $\frac12$-co-Lipschitz fashion. Then the result follows from the co-area formula for Alexandrov spaces (see \cite{AKP}). In fact, we prove that for each $y\in G_Y^m$, there exists $r_0=r_0(y)>0$ such that the following holds for any $0<r\le r_0$:
  $$\bigcup_{x_j\in f^{-1}(y)}f\left(B_r(x_j)\right)
    =f\left(B_r(f^{-1}(y))\right)\supset B_{r/2}(y).
  $$
  Argue by contradiction. Suppose that there exists a sequence $z_i\neq y$, $z_i\to y$, such that $d\left(f^{-1}(y), f^{-1}(z_i)\right)\ge 2|yz_i|$. Let $x_i\in f^{-1}(z_i)$. Passing to a subsequence, assume that $x_i$ converge to a point $x\in X$. Because $f$ is continuous, $x\in f^{-1}(y)$. Let $\delta=1/10$ and $r_0=r(x)$ be selected as in Lemma \ref{bdy.local.iso}. For $i$ large, we have $x_i\in B_{r_0}(x)$ and $x_i\neq x$. Therefore,
  \begin{align*}
    |yz_i|&=|f(x)f(x_i)|\ge (1-\delta)|xx_i|
    \ge (1-\delta)d\left(f^{-1}(y), f^{-1}(z_i)\right)\\&\ge 2(1-\delta)|yz_i|
    >|yz_i|.
  \end{align*}
\end{proof}

\begin{proof}[{\bf Proof of Theorem \ref{thmC.spd.glue}}]
  We see that the map $df_{f^{-1}(p)}\colon \Sigma_{f^{-1}(p)}(X)\to\Sigma_p(Y)$ is 1-Lipschitz onto and volume preserving. Moreover, by Theorem \ref{main.thm} (3), $\Sigma_{f^{-1}(p)}(X)$ consists of finitely many Alexandrov spaces. The desired result is obtained by a direct application of Theorem \ref{main.thm}.
\end{proof}

In the end of this section, we prove the following extra properties for the gluing points.

\begin{prop}\label{spd.glue.cor}\quad
  \begin{enumerate}
    \item $f(X^\delta)\subseteq Y^\delta\setminus G_Y\subseteq f(X^{\tau(\delta)})$ for $\delta>0$ small. In particular, $f(\reg X)=\reg Y\setminus G_Y$.
    \item $\partial Y\subseteq f(\partial X)$.
  \end{enumerate}
\end{prop}

\begin{proof}
  (1) Clearly, $f(X^\delta)\subseteq Y^\delta\setminus G_Y$. For any $y\in Y^\delta\setminus G_Y$, let $x\in X$ such that $f(x)=y$. By Theorem \ref{thmC.spd.glue},
  $$\vol{\Sigma_{x}}=\vol{\Sigma_y}\ge\vol{\Bbb S_1^{n-1}}-\tau(\delta).$$
  By Theorem \ref{abs.max}, we have $x\in X^{\tau(\delta)}$.

  (2) It's equivalent to show that $f(\intl X)\subseteq \intl Y$. Let $x\in \intl X$ and $y=f(x)$. We shall show that $\partial\Sigma_y=\varnothing$. By the gluing of spaces of directions, $\Sigma_{f^{-1}(y)}$ and $\Sigma_y$ satisfy the assumption as in Theorem \ref{main.thm}. Note that $\partial\Sigma_{x}=\varnothing$. By Corollary \ref{thmC.vol.iso}, $\Sigma_y$ is isometric to $\Sigma_{x}$, which has an empty boundary.
\end{proof}



\section{Applications}

Let $p_i\in X\in\Alexnk$ and $p_i\to p$. It is known that $\dsp\liminf_{i\to\infty}\Sigma_{p_i}\ge\Sigma_p$ (Theorem \ref{bgp7.14}). A natural question to ask is, when do we have $\dsp\lim_{i\to\infty}\Sigma_{p_i}=\Sigma_p$? This is true in a special case that $p_i$ are interior points along the same geodesic (Theorem \ref{pet98}). Using the LV-rigidity theorem, we prove the following theorem.





\begin{thm}[LV-rigidity of spaces of directions]\label{shrink.spd}
  Let $X_i\in\Alexnk$. Suppose that $(X_i, p_i)\overset{d_{GH}}{\longrightarrow}(X, p)\in\Alexnk$ and $\dsp\lim_{i\to\infty}\Sigma_{p_i}$ exists. Then $\dsp\lim_{i\to\infty}\Sigma_{p_i}=\Sigma_p$ if and only if  $\dsp\lim_{i\to\infty}\vol{\Sigma_{p_i}}=\vol{\Sigma_p}$.
\end{thm}


\begin{proof}
  If $\dsp\lim_{i\to\infty}\Sigma_{p_i}=\Sigma_p$, then $\dsp\lim_{i\to\infty}\vol{\Sigma_{p_i}}=\vol{\Sigma_p}$ follows from the continuity of volumes in Alexandrov spaces. Now we assume $\dsp\lim_{i\to\infty}\Sigma_{p_i}=\Sigma$ and $\dsp\lim_{i\to\infty}\vol{\Sigma_{p_i}}=\vol{\Sigma_p}$. We shall show that $\Sigma_p$ is isometric to $\Sigma$.

  Let $g\colon \Sigma_p\to\Sigma$ be the distance non-decreasing map defined as in Theorem \ref{bgp7.14}. We repeat the definition here. It's sufficient to define $g$ over the directions where geodesics can go out. Let $v=\drn_{\geod{pq}}$ be a such direction. Take $q_1\in\geod{pq}$ be the midpoint. There is a subsequence of $\drn_{\geod{p_iq_1}}$ (still called $\drn_{\geod{p_iq_1}}$) Gromov-Hausdroff converges to a direction in $\Sigma$. We define $\dsp g(v)=\lim_{i\to\infty}\drn_{\geod{p_iq_1}}$. Note that $g$ may depend on the selection of the subsequence. No matter which subsequence is selected, since $\geod{p_iq_1}\to\geod{pq_1}$, we always have $|g(v_1)g(v_2)|\ge |v_1v_2|$ (\cite{BGP} Lemma 2.8.1).

  For the above $g$, we show that $g(\intl{\Sigma_p})\subseteq\intl\Sigma$. Let $v\in\intl{\Sigma_p}$ be a direction pointing to the interior of $X$. It suffices to show that $|g(v)\xi|_\Sigma>0$ for any direction $\xi\in\partial\Sigma$. By Perel'man's stability theorem, there are $\xi_i\in \partial\Sigma_{p_i}$ such that $\xi_i\to\xi$. For each $i$, let $\gamma_i(t)$ be a quasi-geodesic going out from $p_i$ in the direction $\xi_i$. Because $\partial X_i$ is an extremal subset, we have $\gamma_i\subset\partial X_i$. Let $i\to\infty$. There is a subsequence of $\gamma_i$ (still called $\gamma_i$) that converges to a quasi-geodesic $\gamma\subset\partial X$ going out from $p$. This implies that $\gamma^+(0)\in\partial\Sigma_p$. Let $\epsilon>0$ be small and $q\in\intl X$ so that $|v\drn_{\geod{pq}}|<\epsilon$. Take $q_1\in\geod{pq}$ be the midpoint. Choose a subsequence of $p_i$ such that $\dsp g(v)=\lim_{i\to\infty}\drn_{\geod{p_iq_1}}$. Because $\geod{p_iq_1}\to\geod{pq_1}$ and the geodesic from $p$ to $q_1$ is unique, for $\epsilon>0$ sufficiently small, we have
  \begin{align*}
    |g(v)\xi|_\Sigma &=\lim_{i\to\infty}|\drn_{\geod{p_iq}}\gamma_i^+(0)|
    \\
    &\ge|\drn_{\geod{pq}}\gamma^+(0)|
    \ge d(\drn_{\geod{pq}},\partial\Sigma_p)\ge d(v,\partial\Sigma_p)-\epsilon>0.
  \end{align*}

  In our case, using $\vol{\Sigma_p}=\vol\Sigma$, together with that $\Sigma_p$, $\Sigma$ are both compact, the map $g^{-1}$ can be extended to a 1-Lipschitz onto map $f\colon \Sigma\to\Sigma_p$. Noting that $g(\intl{\Sigma_p})\subseteq\intl\Sigma$, we have $f(\partial\Sigma)\subseteq\partial\Sigma_p$ if $\partial\Sigma\neq\varnothing$. Then by Corollary \ref{thmC.vol.iso}, $\Sigma_p$ is isometric to $\Sigma$.

\end{proof}

Using the LV-rigidity theorem and the above theorem, we are able to classify the Alexandrov spaces which achieve/almost achieve their relatively maximum volume. Let $C_{\kappa}(\Sigma)$ be the $\kappa$-cone (see \cite{BGP} $\S4$) over $\Sigma\in\Alex^{n-1}(1)$. Let  $C_{\kappa}^r(\Sigma)=B_r(O,C_{\kappa}(\Sigma))$ be the $r$-ball in $C_{\kappa}(\Sigma_p)$ centered at the vertex $O$, and $\bar C_{\kappa}^r(\Sigma)$ be its closure. Let $\Sigma\times\{r\}=\{x\in C_{\kappa}(\Sigma): |Ox|=r\}$ denote the $r$-cross section in $C_{\kappa}(\Sigma)$.

\begin{thm}[Relatively maximum volume]\label{rel.max}
Let $p\in X\in \Alexnk$. For any $0<r<R$, if the equality in the Bishop-Gromov relative volume comparison
$$\frac{\vol{B_R(p)}}{\vol{B_r(p)}}
\le\frac{\vol{B_R(o,\Bbb S_{\kappa}^n)}}{\vol{B_r(o,\Bbb S_{\kappa}^n)}}$$
holds, then the metric ball $B_R(p)$ is isometric to $C_{\kappa}^R(\Sigma_p)$ with respect to their intrinsic metrics. If $X=\overline{B_R(p)}$, then
\begin{enumerate}
  \item $R\le\frac{\pi}{2\sqrt\kappa}$ or $R=\frac{\pi}{\sqrt\kappa}$ for $\kappa>0$;
  \item $X$ is isometric to a self-glued space $\bar C^R_\kappa(\Sigma_p)/(a\sim \phi(a))$, where $\phi: \Sigma_p\times \{R\}\to\Sigma_p\times \{R\}$ is an isometric involution;
  \item\footnote{It was kindly pointed out to us by the referee, that in \cite{LR10}, the statement ``if $X$ is a topological manifold, then $X$ is homeomorphic to $\Bbb S^n$ or $\Bbb{RP}^n$" is incorrect.} if $X$ is a topological manifold, then $X$ is homeomorphic to $\Bbb S^n$ or homotopy equivalent to $\Bbb{RP}^n$.
\end{enumerate}
\end{thm}

The space $\Sigma_p\in\Alex^{n-1}(1)$ is relatively determined by the space $X$ since it is the space of direction of $X$ at $p$. The statements (1) and (2) were proved in \cite{LR10} using a technique that heavily relies on the cone structures. Here we give a direct proof using Theorem \ref{main.thm} and provide more details for the proof of (3). In \cite{GP92}, Grove and Petersen proved a stronger result assuming that $X$ is a limit of Riemannian manifolds that satisfies $\vol X=\vol{B_R(o,\Bbb S_\kappa^n)}$. The case $X\in\Alexnk$ with $\vol X= \vol{B_R(o,\Bbb S_\kappa^n)}$ was studied in \cite{Sh}.

To prove Theorem \ref{rel.max}, we need the following two lemmas from \cite{Br}.
\begin{lem}[Theorem 5.5, p.132 in \cite{Br}]\label{br5.5}
  Let $K$ be a $G$-manifold, where $G$ is a finite group. Assume that for a given
prime $p$, all $p$-subgroups $P\subseteq G$ satisfies that
$$H_i(K^P;\Bbb Z_p)=0,\qquad 1\le i\le q \text{ (including $P=\{e\})$}.$$
Then $H_i(K/G;\Bbb Z_p)=0$ for
all $1\le i\le k$. Moreover, if this holds for all prime $p$ and $H_i(K;\Bbb Z)=0$ for $1\le i\le k$, then $H_i(K/G;\Bbb Z)=0$ for $1\le i\le k$.
\end{lem}

\begin{lem}[Exercise 3, p.167 in \cite{Br}]\label{brexe}
  Let $K$ be a finite-dimensional regular $G$-complex, where $G$ is a cyclic group of prime order $p$. Suppose that $K$ is a mod $p$ homology $n$-sphere and that the set of fixed points $K^G$ is a mod $p$ homology $k$-sphere. Then
  $$H_i(K/G;\Bbb Z_p)
    =\begin{cases}
      \Bbb Z_p \quad \text{for} \quad i=0;
      \\
      \Bbb Z_p \quad \text{for} \quad k+2\le i\le n;
      \\
      0 \quad \text{otherwise}.
    \end{cases}
  $$
\end{lem}

\begin{proof}[\bf Proof of Theorem \ref{rel.max}]
By Lemma 4.3 in \cite{LR10}, we see that if the equality holds, then $\vol{B_R(p)}=\vol{C_\kappa^R(\Sigma_p)}$ and the gradient exponential map (\cite{Pet07}) $g\exp_p:C_\kappa^R(\Sigma_p)\to B_R(p)$ is 1-Lipschitz onto. Note that even though $\bar C_\kappa^R(\Sigma_p)$ may not be an Alexandrov space, the cone structure and $\Sigma_p\in\Alex^{n-1}(1)$ guarantee that the proof of Lemma \ref{thmC.int.iso} also applies. Thus we have that $g\exp_p|_{\intl{C_\kappa^R(\Sigma_p)}}$ is an isometry. By the volume condition, the points $a, b\in \bar C_\kappa^R(\Sigma_p)$ can not be glued if $|O_pa|<|O_pb|$, otherwise $\vol{B_{|O_pa|}(p)}>\vol {C_\kappa^{|O_pa|}(\Sigma_p)}$. Consequently, for any $a\in\bar C_\kappa^R(\Sigma_p)$, $g\exp(\geod{O_pa})$ is a geodesic and $d(p,g\exp_p(a))=|O_pa|$.

Now we show that the gluing only happens along $\Sigma_p\times\{R\}$. Suppose that $a,b\in C_\kappa^R(\Sigma_p)$ are glued, but $\geod{Oa}$ and $\geod{Ob}$ are not glued with each other, that is, $g\exp_p(a)=g\exp_p(b)=x$, but $g\exp_p(\geod{Oa})\neq g\exp_p(\geod{Ob})$. Because $|Oa|=|Ob|<R$, we can extend the geodesic $\geod{Oa}$ to a point $a'$ such that $|Oa'|=R$. By the above assertion, $g\exp_p(\geod{Oa'})$ is a geodesic in $X$ and $d(p,g\exp_p(a'))=R$. Note that $\lh{g\exp_p(\geod{O_pb})}+|xy|=|O_pb|+|aa'|=R$. Thus $g\exp_p(\geod{O_pb})\cup g\exp_p(\geod{aa'})$ is also a geodesic. We obtain a bifurcated geodesic at $x$. Now we get that $g\exp_p(\geod{O_pa})= g\exp_p(\geod{O_pb})$, but $\drn_{O_p}^a\neq\drn_{O_p}^b$. This can not happen in Alexandrov spaces.

By the same argument as in \cite{LR10} Lemma 2.6, we see that for any $a, b\in\Sigma_p\times\{R\}$ that satisfy $a\neq b$ and $g\exp_p(a)=g\exp_p(b)=q\in X$,  $g\exp_p\left(\geod{O_pa}\right) \cup g\exp_p\left(\geod{O_pb}\right)$ forms a local geodesic at $q$. Thus $G_Z^m=\varnothing$ for $m\ge 3$, and the gluing happens along an involution of $\Sigma_p\times\{R\}$, that is, every $a\in \Sigma_p\times\{R\}$ is glued with $\phi(a)$, where $\phi: \Sigma_p\times \{R\}\to\Sigma_p\times \{R\}$ is an involution.

In order to apply Theorem \ref{main.thm} and show that $g\exp_p$ preserves the lengths of paths along the boundary, we need to show that $\bar C_\kappa^R(\Sigma_p)$ is an Alexandrov space, which is a consequence of (1). For simplicity, we assume that $\kappa=1$, $\frac\pi2<R<\pi$ and $X\in\Alex^n(1)$. We claim that if $\frac\pi2<R<\pi$, then $\Sigma_p\times\{R\}$ must be identified as one point, which contradicts to that the maximum gluing number $m_0\le 2$. Suppose that $x,y\in g\exp_p(\Sigma_p\times\{R\})$, but $x\neq y$. Clearly $|px|=|py|=R$. Then $\ang{x}{p}{y}>\frac\pi2$ due to the triangle comparison and $\frac\pi2<R<\pi$. Let $\geod{xy}$ be a geodesic connecting $x$ and $y$ in $X$. By the first variation, there is a point $q\in\geod{xy}$ so that $|pq|>|px|=R$. This contradicts to that $X\subseteq\overline{B_R(p)}$.

Proof of (2). Now we can apply Theorem \ref{main.thm} and get that $g\exp_p$ is a path isometry. However, this does not immediately imply that $\phi$ is an isometry. One needs to show that the involution $\phi$ is continuous. Suppose that $a_i\to a$ and $\phi(a_i)=b_i\to b$. We shall prove $\phi(a)=b$. By the continuity of $g\exp_p$, we have $$g\exp_p(a)=\lim_{i\to\infty}g\exp_p(a_i) =\lim_{i\to\infty}g\exp_p(\phi(a_i))=g\exp_p(b).$$
We claim that either $a$ is $\phi$-fixed or $a\neq b$. In the first case, we have $a=b$ and thus $\phi(a)=a=b$. In the latter case, we also have $\phi(a)=b$ because the maximal gluing number is at most $2$. Now we prove the claim in the following two cases.

Case 1, passing to a subsequence, $a_i$ are all $\phi$-fixed points. We show that $a$ is also $\phi$-fixed. Let $x_i=g\exp_p(a_i)\to x=g\exp_p(a)$. Then there are two geodesics $\gamma^+$ and $\gamma^-$ from $p$ to $x$. Moreover, the opposite tangent vectors $v_x^+$, $v_x^-$ at $x$ form an angle $\pi$. For $i$ large, let $y_i^\pm\in\gamma^\pm$ such that $|py_i^\pm|=|xx_i|^{1/4}$. By a rescaling argument, we see that geodesics $\geod{x_iy_i^\pm}$ converge to geodesic $\gamma^\pm$ as $i\to \infty$. Because for each $x_i$, there is a unique geodesic from $p$ to $x_i$, we have $\ang{x_i}{y_i^+}{y_i^-} \le\ang{x_i}{p}{y_i^+}+\ang{x_i}{p}{y_i^-}\le \frac{\pi}{2}$, provided that $|py_i^\pm|$ are sufficiently small. Thus
$$\frac\pi2
  \ge\liminf_{i\to\infty}\ang{x_i}{y_i^+}{y_i^-}
  \ge |v_x^+v_x^-|_{\Sigma_x}=\pi,
$$
a contradiction.

Case 2, passing to a subsequence, $\phi(a_i)\neq a_i$ for all $i$. We show that $a\neq b$, if $a$ is not $\phi$-fixed. Suppose $a=b=s$. Let $x_i=g\exp_p(b_i)=g\exp_p(a_i)$ and $x=g\exp_p(s)$. Clearly $x_i\to x$ and geodesics $g\exp_p(\geod{O_pa_i})$ and $g\exp_p(\geod{O_pb_i})$ both converge to the geodesic $g\exp_p(\geod{Os})$. Choose points $a_i'\in g\exp_p(\geod{O_pa_i})$, $b_i'\in g\exp_p(\geod{O_pb_i})$, $y\in g\exp_p(\geod{O_ps})$ and $z\in g\exp_p(\geod{O_p\phi(s)})$ such that $\geod{x_ia_i'}\to\geod{xy}$, $\geod{x_ib_i'}\to\geod{xy}$ and $\geod{x_iz}\to\geod{xz}$. Because $g\exp_p(\geod{O_ps})$ and $g\exp_p(\geod{O_p\phi(s)})$ form a local geodesic at $x$, we can choose $y,z$ close to $x$ such that $\ang{x}{y}{z}=\pi$. By the construction, when $i$ large, we have
$$\left|\drn_{x_i}^{a_i'}\,\drn_{x_i}^z\right|_{\Sigma_{x_i}}>\pi-\epsilon
  \quad\text{and}\quad
  \left|\drn_{x_i}^{b_i'}\,\drn_{x_i}^z\right|_{\Sigma_{x_i}}>\pi-\epsilon.
$$
Note that $g\exp_p(\geod{O_pa_i})\cup g\exp_p(\geod{O_pb_i})$ also forms a local geodesic at $x_i$. We have $|\drn_{x_i}^{a_i'}\,\drn_{x_i}^{b_i'}|_{\Sigma_{x_i}}=\pi$. Thus
$$\left|\drn_{x_i}^{a_i'}\,\drn_{x_i}^z\right|_{\Sigma_{x_i}}
  +\left|\drn_{x_i}^{b_i'}\,\drn_{x_i}^z\right|_{\Sigma_{x_i}}
  +\left|\drn_{x_i}^{a_i'}\,\drn_{x_i}^{b_i'}\right|_{\Sigma_{x_i}}
  >3\pi-2\epsilon.
$$
For $\epsilon>0$ small, this contracts to $\Sigma_{x_i}\in\Alex^{n-1}(1)$.

Now we prove that the above involution $\phi:\Sigma_p\times\{R\}\to\Sigma_p\times\{R\}$ is an isometry. Let $a_1, b_1\in\Sigma_p\times\{R\}$ and $a_2=\phi(a_1)$, $b_2=\phi(b_1)$. Let $\gamma\subset \Sigma_p\times\{R\}$ be a minimizing geodesic connecting $a_1$ and $b_1$ with respect to the intrinsic metric of $\Sigma_p\times\{R\}$. Since $\phi$ is continuous, $\phi(\gamma)\subset \Sigma_p\times\{R\}$ is a curve connecting $a_2$ and $b_2$. Clearly $g\exp_p(\gamma)=g\exp_p(\phi(\gamma))$. By Theorem \ref{main.thm}, $g\exp_p$ is a path isometry. Thus,
$$|a_1b_1|_{\Sigma_p\times\{R\}}=\lh\gamma=\lh{g\exp_p(\gamma)} =\lh{g\exp_p(\phi(\gamma))}=\lh{\phi(\gamma)}\ge|a_2b_2|_{\Sigma_p\times\{R\}}.$$ Similarly, $|a_2b_2|_{\Sigma_p\times\{R\}}\ge|a_1b_1|_{\Sigma_p\times\{R\}}$. Thus $$|a_1b_1|_{\Sigma_p\times\{R\}}=|a_2b_2|_{\Sigma_p\times\{R\}} =|\phi(a_1)\phi(b_1)|_{\Sigma_p\times\{R\}}.$$



(3) As proved in \cite{LR10}, $X$ can be written as a $\Bbb Z_2$-quotient of $K$, where $K$ is constructed via a gluing of two copies of $\bar C_\kappa^R(\Sigma_p)$ induced by the involution $\phi$. It was shown in \cite{LR10} that $K$ is homeomorphic to the sphere $\Bbb S^n$. If the $\Bbb Z_2$-action is free, then $X=K/\Bbb Z_2$ is homotopy equivalent to $\Bbb{RP}^n$. If the set of the fixed points is not empty, then $X$ is a simply connected topological manifold. This is because the induced map, $\pi_1(K)\to \pi_1(X)$ is onto. If $n\le 3$, then $X$ is homeomorphic to $\Bbb S^n$ by Poincar\'e conjecture.

Assume that $n\ge 4$. Let $k$ be the dimension of the set of $\Bbb Z_2$-fixed points. Clearly, $k\le n-2$, otherwise $X$ has a non-empty boundary. By Smith theorem, the $\Bbb Z_2$-fixed point set $K^{\Bbb Z_2}$ is a $\Bbb Z_2$-homology sphere. Then by Lemma \ref{brexe} and Poincar\'e duality ($X$ is a manifold), we get that $k+2=n$, that is, $X$ is a mod 2 homology sphere.

We claim that $X$ is an integral homology sphere. If $n\ge 5$, then $k=n-2>\frac n2$. By Lemma \ref{br5.5} and Poincar\'e duality, we obtain that $H_i(X;\Bbb Z)=0$ for all $1\le i\le n-1$. When $n=4$, by the same argument, we get that $H_i(X;\Bbb Z)=0$ for $i=1,3$. Because $H_1(X;\Bbb Z)=0$, by the Universal coefficient theorem, $H_2(X;\Bbb Z)\otimes\Bbb Z_2=H_2(X;\Bbb Z_2)=0$. Using $H_1(X;\Bbb Z)=0$ again and by Poincar\'e duality, we get that the torsion subgroup of $H_2(X;\Bbb Z)$ is trivial. Thus $H_2(X;\Bbb Z)=0$.

Since $X$ is simply connected, by Hurewicz theorem, an integral homology sphere is a homotopy sphere. Thus $X$ is homeomorphic to a sphere by Poincar\'e conjecture.
\end{proof}

Using Perel'man's stability theorem, Theorems \ref{shrink.spd} and \ref{rel.max}, we get the following stability theorem, which generalizes the result in \cite{LR10} without assuming that $X$ is a topological manifold.

\begin{thm}[Stability of relatively maximum volume]\label{almost.rel.max}
  Let $\Sigma\in\Alex^{n-1}(1)$ and $n, \kappa, R>0$. There exists a constant $$\epsilon=\epsilon(\Sigma_p,n,\kappa,R)>0$$
  such that if $X\in\Alexnk$ satisfies $X=\overline{B_R(p)}$, $\Sigma_p=\Sigma$ and $\vol X>
  \vol{\bar C_{\kappa}^R(\Sigma)}-\epsilon$ for some $p\in X$, then $X$ is homeomorphic to a self-glued space $\bar C^R_\kappa(\Sigma)/(x\sim \phi(x))$,
  where $\phi: \Sigma\times \{R\}\to\Sigma\times \{R\}$ is an isometric
  involution. In particular, if $X$ is a topological manifold, then $X$ is homeomorphic to $\Bbb S^n$ or homotopy equivalent to $\Bbb{RP}^n$.
\end{thm}

\begin{proof}
  Let $(X_i,p_i)\in\Alexnk$ be a Gromov-Hausdorff convergent sequence that satisfies $X_i=\overline{B_R(p_i)}$ and $\Sigma_{p_i}=\Sigma$ for all $i$. Suppose that $\dsp\lim_{i\to\infty}\vol{X_i}
  =\vol{\bar C_\kappa^R(\Sigma_p)}$. Let $(X, p)$ be the limit space of $(X_i,p_i)$. Then $X=\overline{B_R(p)}$ and $\vol{X}=\vol{\bar C_\kappa^R(\Sigma)}$. Because there is a distance non-decreasing map $g\colon\Sigma_{p}\to\Sigma$, we have
  \begin{align}
    \vol{\bar C_{\kappa}^R(\Sigma_{p})}
    \le\vol{\bar C_{\kappa}^R(\Sigma)}
    =\vol X
    \le\vol{\bar C_{\kappa}^R(\Sigma_{p})}.
    \label{pf.thmD.e1}
  \end{align}
  Thus $\vol{X}=\vol{\bar C_{\kappa}^R(\Sigma_{p})}$. By Theorem \ref{rel.max}, $(X, p)$ is isometric to a self-glued space $\bar C^R_\kappa(\Sigma_{p})/(x\sim \phi(x))$, where $\phi: \Sigma_{p}\times \{R\}\to\Sigma_{p}\times \{R\}$ is an isometric involution. By (\ref{pf.thmD.e1}) again, we see that $\vol{\Sigma_{p}}=\vol{\Sigma}$. Thus $\Sigma_{p}$ is isometric to $\Sigma$ by Theorem \ref{shrink.spd}. Then the theorem follows from Perel'man's stability theorem.
\end{proof}


%

\vskip 30mm

\bibliographystyle{amsalpha}


\end{document}